\documentclass[reqno,11pt]{amsart}
\usepackage{amsmath, latexsym, amsfonts, amssymb, amsthm, amscd}

\usepackage{graphicx}
\usepackage{pstricks}
\numberwithin{equation}{section}

\numberwithin{equation}{section}

\newtheorem{theorem}{Theorem}[section]

\newtheorem{lemma}[theorem]{Lemma}
\newtheorem{prop}[theorem]{Proposition}
\newtheorem{cor}[theorem]{Corollary}
\newtheorem{rem}[theorem]{Remark}

\newtheorem{definition}[theorem]{Definition}

\newcommand{\un}{\mathbf{1}}

\newcommand{\eps}{\epsilon}

\newcommand{\gep}{\varepsilon}       

\newcommand{\cD}{{\ensuremath{\mathcal D}} }
\newcommand{\cF}{{\ensuremath{\mathcal F}} }

\newcommand{\cE}{{\ensuremath{\mathcal E}} }

\newcommand{\cM}{{\ensuremath{\mathcal M}} }
\newcommand{\cN}{{\ensuremath{\mathcal N}} }

\newcommand{\bY}{{\ensuremath{\mathbf Y}} }

\newcommand{\bbE}{{\ensuremath{\mathbb E}} }
\newcommand{\E}{{\ensuremath{\mathbb E}} }

\newcommand{\N}{{\ensuremath{\mathbb N}} }

\renewcommand{\P}{{\ensuremath{\mathbb P}} }
\newcommand{\bbP}{{\ensuremath{\mathbb P}} }
\newcommand{\Q}{{\ensuremath{\mathbb Q}} }
\newcommand{\bbQ}{{\ensuremath{\mathbb Q}} }

\newcommand{\R}{{\ensuremath{\mathbb R}}}

\newcommand{\ZV}{Z}

\newfont{\indic}{bbmss12}

\def\un#1{\hbox{{\indic 1}$_{#1}$}}

\title[Fleming-Viot Processes with Mutations]{On Exceptional Times for generalized Fleming-Viot Processes with Mutations}

\author{J. Berestycki}
\thanks{}
\address{Laboratoire de Probabilit\'es et Mod\'eles Al\'eatoires Universit\'e Paris 6, 4 place Jussieu, 75252 Paris Cedex 05, France}
\email{$\{$julien.berestycki, leif.doering, lorenzo.zambotti$\}$@upmc.fr}

\author{L. D\"oring}

\author{L. Mytnik}
\address{Faculty of Industrial Engineering and Management Technion Israel Institute of Technology, Haifa 32000, Israel}
\email{leonid@ie.technion.ac.il}

\author{L. Zambotti}

\subjclass[2000]{Primary 60J80; Secondary 60G18}
\keywords{Fleming-Viot Processes, Mutations, Exceptional Times, Excursion Theory, Jump-Type SDE, Self-Similarity}

\begin{document}

\begin{abstract}

If $\mathbf Y$ is a standard Fleming-Viot process with constant mutation rate (in the infinitely many sites model) then it is well known that for each $t>0$ the measure  $\mathbf Y_t$ is purely atomic with infinitely many atoms. However, Schmuland proved that there is a critical value for the mutation rate under which almost surely there are exceptional times at which $\mathbf Y$ is a finite sum of weighted Dirac masses.
In the present work we discuss the existence of such exceptional times for the generalized Fleming-Viot processes. In the case of Beta-Fleming-Viot processes 
with index $\alpha\in\,]1,2[$ we show that - irrespectively of the mutation rate and $\alpha$ - the number of atoms is almost surely always infinite. 
The proof  combines a Pitman-Yor type representation with a disintegration formula, Lamperti's transformation for self-similar processes and covering results for Poisson point processes. 

\end{abstract}
\maketitle
\section{Main Result}
The measure-valued {\it Fleming-Viot diffusion processes} were first introduced by Fleming and Viot \cite{FlemingViot} and have become a cornerstone of mathematical population genetics in the last decades. It is a model which describes the evolution (forward in time) of the genetic composition of a large population. Each individual is characterized by a {\it genetic type} which is a point in a type-space $E$. The Fleming-Viot process is a Markov process $(\textbf Y_t)_{t\geq 0}$ on
	\begin{align*}
		\mathcal M^1_{E}=\big\{\nu: \nu\text{ is a probability measure on } E\big\}
	\end{align*}
for which we interpret $\mathbf Y_t(B)$ as the proportion of the population at time $t$ which carries a genetic type belonging to a Borel set $B$ of types. In particular, the number of (different) types at time $t$ is equal to the number of atoms of $\textbf Y_t$ with the convention that the number of types is infinite if $\textbf Y_t$ has absolutely continuous part.

Fleming-Viot superprocesses can be defined through their infinitesimal generators
\begin{equation}\label{GeneratorFV}
(\mathcal{L} \phi)(\mu) = \int_E\int_E \mu(dv) (\delta_v(dy)-\mu(dy)) {\delta^2 \phi(\mu) \over \delta \mu(v) \delta \mu(y)} +\int_E \mu(dv) A\left( {\delta \phi (\mu) \over \delta \mu(\cdot) }\right)(v),
\end{equation}
acting on smooth test-functions where $\delta \phi(\mu) /\delta \mu(v) = \lim_{\epsilon \to 0+} \epsilon^{-1}\{ \phi(\mu+\epsilon \delta_v) -\phi(\mu))$ and $A$ is the generator for a Markov process in $E$ which represents the effect of mutations. Here $\delta_v$ is the Dirac measure at $v$. It is well known that the Fleming-Viot superprocess arises as the scaling limit of a Moran-type model for the evolution of a finite discrete population of fixed size if the reproduction mechanism is such that no individual gives birth to a positive proportion of the population in a small number of generations.
 For a detailed description of Fleming-Viot processes and discussions of variations we refer to the overview article of Ethier and Kurtz \cite{EthierKurtz93} and to Etheridge's lecture notes \cite{etheridge lect}. 

The first summand of the generator reflects the genetic resampling mechanism whereas the second summand represents the effect of mutations. Several choices for $A$ have appeared in the literature. In the present work we shall work in the setting of the {\it infinite site model} where each mutation creates a new type never seen before. Without loss of generality
let the type space be $E=[0,1]$.
Then the following choice of $A$  gives an example of an infinite site model with mutations: 
	\begin{equation}\label{E:IAM mutation operator}
		(Af)(v) = \theta \int_E (f(y)-f(v))dy,
	\end{equation}
for some $\theta >0$. 
The choice of the uniform measure $dy$ is arbitrary (we could choose the new type according to any distribution that has a density with respect to the Lebesgue measure), all that matters is that the newly created type $y$ is different from all other types. With
 $A$ as in~(\ref{E:IAM mutation operator}), mutations arrive at rate $\theta$ and create a new type picked at random from $E$ according to the uniform measure, therefore the corresponding process is sometimes 
called the Fleming-Viot process with neutral mutations.  
 
\medskip

Let us briefly recall two classical facts concerning the infinite types Fleming-Viot process described above  (for a more complete picture we refer to the monograph of Etheridge \cite{Etheridge}) for the uniform initial condition $\textbf Y_0$:
\begin{itemize}
\item[(i)] If there is no mutation, then, for all $t>0$ fixed, the number of types is almost surely finite.
\item[(ii)] If the mutation parameter $\theta$ is strictly positive, then, for all $t>0$ fixed,  the number of types is infinite almost surely.
\end{itemize}
A beautiful complement to (i) and (ii) was found by Schmuland for exceptional times that are not fixed in advance:
\begin{theorem}[Schmuland \cite{Schmuland}]\label{schmu} 
		\begin{align*}
		\P\big( \exists \, t>0 : \# \{\text{types at time }t \}<\infty\big)=\left\{ \begin{array}{c}1 \quad\text{ if } \theta <1, \\ 0\quad \text{ if } \theta \ge 1. \end{array}\right. 
	\end{align*}  
\end{theorem}
Schmuland's proof of the dichotomy is based on analytic arguments involving the capacity of finite dimensional subspaces of the infinite dimensional state-space. In Section \ref{sec:schmuland} we reprove Schmuland's theorem via excursion theory. 

In the series of articles \cite{BertoinLeGall1}, \cite{BertoinLeGall2}, \cite{BertoinLeGall3}, Bertoin and Le Gall  introduced and started the study of {\it $\Lambda$-Fleming-Viot processes}, a class of stochastic processes which naturally extends the class of  standard Fleming-Viot processes. These processes are completely characterized by a finite measure $\Lambda$ on $[0,1]$ and a generator $A$. 
Similarly to the standard Fleming-Viot process, these processes can be defined through their infinitesimal generator
\begin{equation}\label{E:GFVmutation generator}
\begin{split}
	(\mathcal{L} \phi)(\mu) = & \int_0^1 y^{-2} \Lambda(dy) \int \mu(da) (\phi((1-y)\mu +y \delta _a)- \phi(\mu)) \\ & + \int_E \mu(dv) A\left( {\delta \phi (\mu) \over \delta \mu(\cdot) }\right)(v),
	\end{split}
\end{equation}
 and the sites of atoms are again called types.
 For $A=0$, the generator formulation only appeared implicitly in \cite{BertoinLeGall2} and is explained in more details in Birkner et al. \cite{7} and for $A$ as in \eqref{E:IAM mutation operator} it can be found in Birkner et al. \cite{Birkner}.
The dynamics of a generalized Fleming-Viot process $(\textbf Y_t)_{t\geq 0}$ are as follows: at rate $y^{-2}\Lambda(dy)$ a point $a$ is sampled at time $t>0$ according to the probability measure $\bY_{t-}(da)$ and a point-mass $y$ is added at position $a$ while scaling the rest of the measure by $(1-y)$ to keep the total mass at 1. The second term of \eqref{E:GFVmutation generator} is the same mutation operator as in \eqref{GeneratorFV}.
 For a detailed description of $\Lambda$-Fleming-Viot processes and discussions of variations we refer to the overview article of Blath and Birkner \cite{BirknerBlath}.

 \smallskip
In the following we are going to focus only on the choice $\Lambda=Beta(2-\alpha,\alpha)$, the Beta distribution with density 
\begin{align*}
	f(u)=	C_\alpha u^{1-\alpha}(1-u)^{\alpha-1}du,\quad\quad C_\alpha=\frac{1}{\Gamma(2-\alpha)\Gamma(\alpha)},
\end{align*}
for $\alpha\in\, ]1,2[$, and mutation operator $A$ as in~\eqref{E:IAM mutation operator}.   The corresponding $\Lambda$-Fleming-Viot process $(\bY_t)_{t\geq 0}$ is called Beta-Fleming-Viot process or $(\alpha,\theta)$-Fleming-Viot process 
and several results have been established in recent years.
The $(\alpha,\theta)$-Fleming-Viot processes converge weakly to the standard Fleming-Viot process as $\alpha$ tends to $2$.  It was shown in \cite{7} that a $\Lambda$-Fleming-Viot process with $A=0$ is related to measure-valued branching processes in the spirit of Perkin's disintegration theorem precisely if $\Lambda$ is a Beta distribution (this relation is recalled and extended in Section~\ref{Sec:Transfer} below).

If we chose $\alpha\in\, ]1,2[$ and $\textbf Y_0$ uniform on $[0,1]$, then we find the same properties (i) and (ii) for the one-dimensional marginals $\textbf Y_t$ unchanged with respect to the classical case \eqref{GeneratorFV},\eqref{E:IAM mutation operator}.
In fact, for a general $\Lambda$-Fleming-Viot process, (i) is equivalent to the requirement that the associated $\Lambda$-coalescent comes down from infinity (see for instance \cite{bbl3}).
 Here is our main result: contrary to  Schmuland's result,
$(\alpha,\theta)$-Fleming-Viot processes with $\alpha\in\, ]1,2[$ and 
$\theta>0$ never have exceptional times:

\begin{theorem} \label{T:main} Let $(\bY_t)_{t\geq 0}$ be an  $(\alpha,\theta)$-Fleming-Viot superprocess with mutation rate $\theta>0$ and parameter $\alpha\in\, ]1,2[.$ If $\textbf Y_0$ is uniform, then
	\begin{align*}
		\P\big( \exists \, t>0 : \# \{\text{types at time }t \}<\infty\big)=0
	\end{align*}
	for any $\theta>0$.
\end{theorem}

One can get a first rough understanding of why this should be true by the following heuristic: Kingman's coalescent comes down from infinity at speed $2/t$, i.e. if $N_t$ is the number of blocks at time $t$ then $N_t \sim 2/t$ almost surely when $t\to0$. It is known that  (see \cite{BertoinLeGall2} or more recently \cite{labbe}) the process $(N_t, t\ge0)$ has the same law as the process of the number of atoms of the Fleming-Viot process. For a Beta-coalescent with parameter $\alpha\in(1,2)$ we have $N_t \sim c_\alpha t^{-1/(\alpha-1)}$ almost surely as $t\to 0$ (see \cite[Theorem 4]{bbs2}). Therefore Kingman's coalescent comes down from infinity much quicker than Beta-coaelscents. Since the speed at which the generalized Fleming-Viot processes looses types roughly corresponds to the speed at which the dual coalescent comes down from infinity, it is possible that $(\alpha,\theta)$-Fleming-Viot processes do not loose types fast enough,
and hence there are no exceptional times at which the number of types is finite.


\section{Auxiliary Constructions}

To prove Theorem \ref{T:main} we construct two auxiliary objects: a particular measure-valued branching process and a corresponding Pitman-Yor type representation. Those will
be used in Section \ref{111} to relate the question of exceptional times to covering results for point processes. In this section we give the definitions and state their relations to the Beta-Fleming-Viot
processes with mutations.\\
All appearing stochastic processes and random variables will be defined on a common stochastic basis $(\Omega, \mathcal G, \mathcal G_t, \P)$ that is rich enough to carry all Poisson point processes (PPP in short) that appear in the sequel.

\subsection{Measure-Valued Branching Processes with Immigration}\label{Sec:existence}
We recall that a continuous state branching process (CSBP in short) with $\alpha$-stable branching mechanism, $\alpha\in\,]1,2]$, is a Markov family $(P_v)_{v\geq 0}$ of probability measures on c\`adl\`ag trajectories with values in $\R_+$, such that  
 \begin{equation}\label{lt}
	E_v\big(e^{-\lambda X_t}\big) = e^{-v\,u_t(\lambda)},\qquad v\geq 0, \lambda \geq 0,
\end{equation}
where for $\psi:\R_+\mapsto\R_+$, $\psi(u):=u^\alpha$, we have the evolution equation
\[
	u'_t(\lambda) = - \psi(u_t(\lambda)),	\hspace{.3in} u_0(\lambda) = \lambda.
\]
For $\alpha=2$, $\psi(u)=u^2$ is the branching mechanism for Feller's branching diffusion, where $P_v$ is the law of the unique solution
to the SDE 
\begin{equation}\label{Feller}
	X_t=v+\int_0^t\sqrt{2X_s}\,dB_s, \qquad t\geq 0,
\end{equation}
driven by a Brownian motion $(B_t)_{t\geq 0}$. On the other hand, for $\alpha\in\, ]1,2[$, $\psi(u)=u^\alpha$ gives the so-called $\alpha$-stable branching processes which can be defined as the unique strong solution of the SDE
\begin{equation}\label{4b}
	X_t=v+\int_0^tX_{s-}^{1/\alpha}\,dL_s, \qquad t\geq 0,
\end{equation}
driven by a spectrally positive $\alpha$-stable L\'evy process $(L_t)_{t\geq 0}$, with L\'evy measure given by 
\[
\un{(x>0)}\, c_\alpha \, x^{-1-\alpha}\,dx, \qquad c_\alpha:=\frac{\alpha(\alpha-1)}{\Gamma(2-\alpha)}.
\]
Note that strong existence and uniqueness for~\eqref{4b} 
follows from the fact that the function $x\mapsto x^{1/\alpha}$ is Lipschitz outside zero, and hence 
strong existence and uniqueness holds for~\eqref{4b} until $X$ hits zero. 
Moreover $X$, being a non-negative martingale,  stays at zero forever after hitting it.
 For a more extensive discussion on strong solutions 
for jumps SDEs see~\cite{FuLi} and~\cite{LiMytnik}.

The main tool that we introduce is a particular {\it measure-valued branching process with interactive immigration} (MBI in short). For a textbook treatment of this subject we refer to Li \cite{Li}. Following Dawson and Li \cite{DawsonLi2}, we are not going to introduce the MBIs via their infinitesimal generators but as strong solutions of a system of stochastic differential equations instead. 
 On $(\Omega, \mathcal G, \mathcal G_t,\P)$, let us consider a Poisson point process
 $\mathcal N=(r_i,x_i,y_i)_{i\in I}$ on $(0,\infty)\times (0,\infty)\times (0,\infty)$ adapted  to 
$ \mathcal G_t$ and with intensity measure 
\begin{equation}\label{nu}
\nu(dr,dx,dy):=\un{(r>0)} \, dr\, \otimes \, c_\alpha \, \un{(x>0)}\, x^{-1-\alpha}\,dx \, \otimes \, \un{(y>0)}\, dy. 
\end{equation}
Throughout the paper we adopt the notation
\[
\tilde\cN:=\cN-\nu,
\]
i.e. $\tilde\cN$ is the compensated version of $\cN$.
It was shown in \cite{DawsonLi2} that the solution to \eqref{4b} 
has the same law as the unique strong solution to the SDE
\begin{equation}\label{4}
		X_t= X_0+\int_{]0,t]\times\R_+\times\R_+} \un{(y<X_{r-})}\, x  \, \tilde\cN(dr,dx,dy)
\end{equation}
with $X_0=v$. 

Now we are going to switch to the measure-valued setting. The real-valued process  $X$ in~\eqref{4b}, \eqref{4} describes the evolution of the {\it total mass} of the CSBP starting at time zero at the mass $X_0=v$. We are going to consider all initial masses $v\in[0,1]$ simultaneously, constructing a process $(\mathbf{X}_t)_{t\geq 0}$ taking values in the space $\cM^{F}_{[0,1]}$ of finite measures on $[0,1]$, endowed with the narrow topology. Assume that at time $t=0$, $\mathbf{X}_0$ is a finite measure on $[0,1]$ with cumulative distribution function $(F(v), v\in[0,1])$, and denote 
$$ X_t(v):= \mathbf{X}_t([0,v]), \quad t\geq 0, v\in[0,1].$$
Then the measure-valued branching process $(\mathbf{X}_t)_{t\geq 0}$ can be constructed in such a way that for each $v$, $(X_t(v))_{t\geq 0}$ 
solves~\eqref{4} with~$X_0=F(v)$, and with the same driving noise for all $v\in[0,1]$. 
In what follows, we deal with a version of \eqref{4} including an immigration term only depending on  the total-mass $X_t(1)$:
\begin{equation}\label{(2.7)} 
	\begin{cases} 
		X_t(v) = F(v) +\int_{]0,t]\times\R_+\times\R_+} \un{(y<X_{r-}(v))}\, x  \, \tilde\cN(dr,dx,dy) + I(v) \int_0^t  g(X_{s}(1))\, ds, 
	\\	v \in [0,1], \ t\geq 0,
	\end{cases}
\end{equation}
where $(I(v), v\in[0,1])$ is the cumulative distribution function of a finite measure on $[0,1]$ and we assume 
\begin{enumerate}
\item[\textbf{(G)}]  $g : \R_+ \mapsto \R_+$ is monotone non-decreasing, continuous and
locally Lipschitz continuous away from zero. 
\end{enumerate} 
\begin{definition}\label{def:2.1}
	An $\cM^{F}_{[0,1]}$-valued  process $(\textbf X_t)_{t\geq 0}$ 
  on $(\Omega, \mathcal G, \mathcal G_t,\P)$
 is called a solution to \eqref{(2.7)} if
	\begin{itemize}
		\item it is c\`adl\`ag $\P$-a.s.,
		\item for all $v\in [0,1]$, setting $X_t(v):=\textbf X_t([0,v])$,  $(X_t(v))_{v\in[0,1],t\geq 0}$ satisfies   $\P$-a.s.
\end{itemize}
	Moreover, a solution $(\textbf X_t)_{t\geq 0}$ is strong if it is adapted to the natural filtration $\cF_t$ generated by $\mathcal N$.
	Finally, we say that pathwise uniqueness holds if 
	\begin{align*}
		\P\big(\textbf X^1_t=\textbf X^2_t, \ \forall t\geq 0\big)=1,
	\end{align*}
	for any two solutions $\textbf X^1$ and $\textbf X^2$ on $(\Omega, \mathcal G, \mathcal G_t,\P)$ driven by the same Poisson point process.
\end{definition}
Here is a well-posedness result for \eqref{(2.7)}:

\begin{theorem}\label{2.2} Let $F$ and $I$ be as above. 
	For any immigration mechanism $g$ satisfying Assumption \textbf{(G)},  there is a strong solution $(\textbf X_t)_{t\geq 0}$ to \eqref{(2.7)} and pathwise uniqueness holds until $T_0:= \inf\{t \ge 0 : \textbf{X}_t([0,1]) =0\}$. 
	\end{theorem}

The proof of Theorem \ref{2.2} relies on ideas from recent articles on pathwise uniqueness for jump-type SDEs such as Fu and Li \cite{FuLi} or Dawson and Li \cite{DawsonLi2}.
Our equation \eqref{(2.7)} is more delicate since all coordinate processes depend on the total-mass $X_t(1)$. The uniqueness statement is first deduced 
for the total-mass $(X_t(1))_{t\geq 0}$ and then for the other coordinates interpreting the total-mass as random environment. 
To construct a (weak) solution we use a (pathwise) Pitman-Yor type representation as explained in the next section.

\subsection{A Pitman-Yor Type Representation for Interactive MBIs}\label{Sec:RayKnight}

Let us denote by $\cE$ the set of c\`adl\`ag trajectories $w:\R_+\mapsto\R_+$ such that $w(0)=0$, $w$ is positive on a bounded interval $]0,\zeta(w)[$ and
$w\equiv 0$ on $[\zeta(w),+\infty[$. 
We recall the construction of the excursion measure of the $\alpha$-stable CSBP $(P_v)_{v\geq 0}$, also called the {\it Kuznetsov measure}, 
see \cite[Section 4]{Li02} or \cite[Chapter 8]{Li}:
For all $t\geq 0$, let $K_t(dx)$ be the unique $\sigma$-finite measure on $\R_+$ such that
\[
\int_{\R_+} \left(1-e^{-\lambda\, x}\right) K_t(dx) = u_t(\lambda) = \left( \lambda^{1-\alpha} + (\alpha-1)t\right)^{\frac1{1-\alpha}}, \qquad  \lambda\geq 0,
\]
where we recall that the function $(u_t(\lambda))_{t\geq 0}$ is the unique solution to the equation
\[
u_t(\lambda) + \int_0^t (u_s(\lambda))^\alpha \, ds = \lambda, \qquad t\geq 0, \ \lambda\geq 0.
\]
We also denote by $Q_t(x,dy)$ the Markov transition semigroup of $(P_v)_{v\geq 0}$. Then there exists a unique Markovian $\sigma$-finite measure $\bbQ$ on $\cE$
with entrance law $(K_t)_{t\geq 0}$ and transition semigroup $(Q_t)_{t\geq 0}$, i.e. such that for all $0< t_1< \cdots<t_n$, $n\in\N$,
\begin{equation}\label{bbQ}
\begin{split}
& \bbQ(w_{t_1}\in dy_1, \ldots, w_{t_n}\in dy_n, \, t_n<\zeta(w)) 
\\ & = K_{t_1}(dy_1) \, Q_{t_2-t_{1}}(y_{1}, dy_2) \cdots Q_{t_n-t_{n-1}}(y_{n-1}, dy_n).
\end{split}
\end{equation}
By construction
\begin{equation}\label{laplace}
  \int_{\cE} \left(1-e^{-\lambda\, w_s}\right) \, \bbQ(dw) = u_s(\lambda) = \left( \lambda^{1-\alpha} + (\alpha-1)s\right)^{\frac1{1-\alpha}}, \qquad s\geq 0, \ \lambda\geq 0,
\end{equation}
and under $\bbQ$, for all $s>0$, conditionally on $\sigma(w_r, r\leq s)$, $(w_{t+s})_{t\geq 0}$ has law $P_{w_s}$. The $\sigma$-finite measure $\bbQ$ is called the {\it  excursion measure}  of the CSBP \eqref{4b}. By \eqref{laplace}, it is easy to check that for any $s>0$
\begin{equation}
\label{Qmom}
\int_\cE w_{s} \, \bbQ(dw) = \left. \frac{\partial}{\partial\lambda} u_s(\lambda) \right|_{\lambda=0} = \lim_{\lambda\downarrow 0} 
\left( 1 + \lambda^{\alpha-1}(\alpha-1)s\right)^{\frac\alpha{1-\alpha}} = 1.
\end{equation}
In Duquesne-Le Gall's setting \cite{LeGallDuquesne}, under the $\sigma$-finite measure $\bbQ$ with infinite total mass, $w$ has the distribution of $(\ell^a(e))_{a\geq 0}$ under $n(de)$, where $n(de)$ is the excursion measure of the height process $H$ and $\ell^a$ is the local time at level $a$. For the more general superprocess setting see for instance Dynkin and Kuznetsov \cite{DK}. 

\smallskip
We need now to extend the space of excursions as follows:
\[
\cD:=\{w:\R_+\mapsto\R_+: \ \exists s\geq 0, \, w\equiv 0 \ {\rm on} \ [0,s], \
w_{\cdot-s}\in \cE\},
\]
i.e. $\cD$ is the set of c\`adl\`ag trajectories $w:\R_+\mapsto\R_+$ such that $w$ is equal to 0 on $[0,s(w)]$, $w$ is positive on a bounded interval $]s(w),s(w)+\zeta(w)[$ and $w\equiv 0$ on $[s(w)+\zeta(w),+\infty[$. For $s\geq 0$, we denote by $\bbQ_s(dw)$ the $\sigma$-finite measure on $\cD$ given by 
\begin{equation}\label{Q_s}
\int_\cD \Phi(w) \, \bbQ_s(dw) := \int_\cE \Phi\left(\un{(\cdot\geq s)}\, w_{\cdot-s}\right) \, \bbQ(dw),
\end{equation}
i.e. $\bbQ_s$ is the image measure of $\bbQ$ under the map 
\begin{equation}\label{gamma}
w\mapsto \gamma_t := \un{(t\geq s)} \, w_{t-s}, \qquad t\geq 0.
\end{equation}
Let us consider a Poisson point process $(s_i, u_i,a_i,w^i)_{i\in I}$ on 
$\R_+\times\R_+\times\cD$ with intensity measure 
\begin{equation}\label{Gamma}
\Gamma(ds,du,da,dw) := \left(\delta_0(ds)\otimes \delta_0(du) \otimes F(da) +
ds\otimes du \otimes I(da)\right)\otimes \bbQ_s(dw)
\end{equation}
where $F$ and $I$ are the cumulative distribution functions appearing in \eqref{(2.7)}. An atom  $(s_i, u_i,a_i,w^i)$ is a population that has immigrated at time $s_i$ whose size evolution is given by $w^i$ and whose genetic type is given by $a_i$. The coordinate $u_i$ is used for thinning purposes, to decide wether or not this particular immigration really happened or not.

\begin{theorem}\label{2.4}
Suppose $g:\R_+\mapsto\R_+$ satisfies Assumption \textbf{(G)}.
Then, for all $v\in[0,1]$, there is a unique c\`adl\`ag process $(Z_t(v), t\geq 0)$ on $(\Omega, \mathcal G, \mathcal G_t, \P)$ satisfying $\P$-a.s.
\begin{equation}\label{Z}
\begin{cases} 
Z_t(v) = \sum_{s_i=0} w_{t}^i \, \un{(a_i \le v)} +
\sum_{s_i>0} w_{t}^i \,\un{(a_i \le v)} \un{(u_i \le g(Z_{s_i-}(1)))},\qquad t> 0,\\
Z_0(v)=F(v).
\end{cases}
\end{equation}
Moreover, we can construct on $(\Omega, \mathcal G, \mathcal G_t, \P)$ a PPP $\cN$ with intensity $\nu$ given by \eqref{nu} such that $Z$ solves \eqref{(2.7)} with respect to $\mathcal N$.
\end{theorem}

If $I(1)=1$, then in the special case of branching mechanism $\psi(\lambda)=\lambda^2$ and constant immigration rate  $g\equiv \theta$, the total-mass process $X_t=X_t(1)$ for \eqref{(2.7)} also solves 
\begin{align*}
\begin{cases} 
	dX_t=\sqrt{2 X_t}\,dB_t+\theta \,dt,\qquad t\geq 0,\\
X_0=F(1).
\end{cases} 
\end{align*}
for which Pitman and Yor obtained the excursion representation in their seminal paper \cite{PY}.
\begin{rem} {\rm
	The recent monograph \cite{Li} by Zenghu Li contains a full theory of this kind of
	Pitman-Yor type representations for measure-valued branching processes, see in particular Chapter 10. We present a different approach below which shows directly how the 
different Poisson point processes in \eqref{(2.7)} and in \eqref{Z} are related to each other. The most important feature of our construction is that it relates
the excursion construction and the SDE construction on a pathwise level.
	}
\end{rem}

Observe that an immediate and interesting corollary of Theorem \ref{2.4} is the following:
\begin{cor}\label{C: pure atom}
Let $g$ be an immigration mechanism satisfying assumption \textbf{(G)} and let $(\mathbf{X}_t)_{t\ge 0}$ be a  solution to  \eqref{(2.7)}. Then almost surely, $\mathbf{X}_t$ is purely atomic for all $t\ge0$.
\end{cor}

In the proof of our Theorem \ref{T:main} we make use of the fact that the Pitman-Yor type representation is well suited for comparison arguments. If $g$ can be bounded from above or below by a constant, then the righthand side of \eqref{(2.7)} can be compared to an explicit PPP for which general theory can be applied.

\subsection{From MBI  to Beta-Fleming-Viot Processes with Mutations}\label{Sec:Transfer}
Let us first recall an important characterization started in \cite{BertoinLeGall2} and completed in \cite{DawsonLi2} which relates Fleming-Viot processes, defined as measure-valued Markov processes by the generator \eqref{E:GFVmutation generator}, and strong solutions to stochastic equations. 
\begin{theorem}[Dawson and Li \cite{DawsonLi2}]\label{T:DL}
Let  $\Lambda$ be the Beta distribution with parameters $(2-\alpha, \alpha)$. 
Suppose $\theta\geq 0$ and $\mathcal M$ is a non-compensated Poisson point process on $(0,\infty)\times [0,1]\times [0,1]$ with intensity $ds \otimes y^{-2} \Lambda(dy) \otimes du$. Then there is a unique strong solution $(Y_t(v))_{t\geq 0,v\in[0,1]}$ to
\begin{equation}\label{FVI} 
	\begin{cases} 
		Y_t(v) = v + \int_{]0,t]\times[0,1]\times[0,1]} y \left[  \un{(u \le Y_{s-}(v))} -Y_{s-}(v) \right] \cM(ds, dy, du) +\theta \int_0^t [v-Y_{s}(v)]ds,\\
		v \in[0,1], \ t\ge 0,
	\end{cases}
\end{equation}
and the measure-valued process $\textbf Y_t([0,v]):={Y_t(v)}$ is an
$(\alpha,\theta)$-Fleming-Viot process  started at uniformly distributed initial condition. 
\end{theorem}
Existence and uniqueness of solutions for this  equation was proved in Theorem 4.4 of \cite{DawsonLi2} while the characterization of the generator of the measure-valued process $\bY$ is the content of their Theorem 4.9. 
\smallskip

We next extend a classical relation between Fleming-Viot processes and measure-valued branching processes which is typically known as disintegration formula. Without mutations, for the standard Fleming-Viot process this goes back to Konno and Shiga \cite{KonnoShiga} and it was shown in Birkner et al. \cite{7} that the relation extends to the generalized $\Lambda$-Fleming-Viot processes without immigration 
if and only if $\Lambda$ is a Beta-measure.
 Our extension relates $(\alpha,\theta)$--Fleming-Viot processes to \eqref{(2.7)} with immigration mechanism $g(x)=\theta x^{2-\alpha}$ and for $\theta=0$ gives an SDE formulation of the main result of \cite{7}.

\begin{theorem}\label{2.7} Let $F(v)=I(v)=v$ and let $g:\R_+\mapsto\R_+$ be defined by $g(x)=\theta x^{2-\alpha}$ for some $\alpha \in(1,2).$ Let then $(\mathbf X_t)_{t\geq 0}$ be the unique solution of  to \eqref{(2.7)} (in the sense of Definition \ref{def:2.1}) such that 
\[
X_t(1)=0, \quad \forall \ t\geq T_0:=\inf\{ s>0: \ X_s(1)=0\}.
\]

Define
\begin{align*}
	S(t) =\alpha(\alpha-1)\Gamma(\alpha) \int_0^tX_s(1)^{1-\alpha}\,ds
\end{align*}
and 
	\begin{equation}\label{E: from Y to X via SDE}
		\textbf Y_t(dv)= \frac{\textbf X_{S^{-1}(t)}(dv)}{X_{S^{-1}(t)}(1)},\quad t\geq 0.
	\end{equation}
	Then $\big(\textbf Y_t\big)_{t\geq 0}$ is well-defined, i.e. $S^{-1}(t)<T_0$ for all $t\geq 0$, and is an $(\alpha,\theta)$-Fleming-Viot  process, i.e. a strong solution to \eqref{FVI} with $\Lambda=Beta(2-\alpha,\alpha)$. 
\end{theorem}
The proof of the theorem is different from the known result for $\theta=0$. To prove that $X_{S^{-1}(t)}(1)>0$ for all $t\geq 0$, Lamperti's representation for CSBPs was crucially used in \cite{7}. This idea breaks down in our generalized setting since the total-mass process $X_t(1)$ is not a CSBP. Our proof uses instead the fact that for all $\theta\geq 0$ the total-mass process is self-similar and an interesting cancellation effect of Lamperti's transformation for self-similar Markov processes and the time-change $S$. \\

In \cite{BDMZ} we study (a generalized version of) the total mass process $(X_t(1), t\ge 0)$ and we show that the extinction time $T_0 = \inf \{ t\ge 0 : X_t(1)=0\}$ is finite almost surely if and only if $\theta <\Gamma(\alpha)$. Otherwise $T_0=\infty$ almost surely. We will see in the proof of Theorem \eqref{2.7} that in both cases $$\lim_{t\to \infty} S^{-1}(t) =T_0 \, \text{ a.s.}$$ 
Theorem \ref{2.7} thus gives some partial information on the behavior of $\big(\textbf X_t\big)_{t\geq 0}$ near the extinction time $T_0$:
\begin{cor}
\label{cor}
As $t \to \infty$ the probability-valued process $\left({\mathbf{X}_{S^{-1}(t)}(dv) \over X_{S^{-1}(t)}(1)}\right)_{t\ge 0}$ converges weakly to the unique invariant measure of $(\mathbf{Y}_t, t\ge 0)$. 

As $t\to T_0$, almost surely, there exists a (random) sequence of times $t_1 <t_2 <\ldots <T_0$ tending to $T_0$ such that the sets
$$
A_i = \text{support of } \mathbf{X}_{t_i}
$$
are pairwise disjoints.
\end{cor}

The first part is a direct consequence of Theorem \ref{2.7} and of the convergence of the $(\alpha,\theta)$- Fleming-Viot process $(\mathbf Y_t, t\ge0)$ to its unique invariant measure. The second part is  a straightforward application of the so-called lookdown representation of $(\mathbf Y_t, t\ge0)$. A sketch of the proof is given in Section \ref{sec:7}.

\section{Proof of Theorems \ref{2.2} and \ref{2.4}}

Recall that $(s_i, u_i,a_i,w^i)_{i\in I}$ is a Poisson point process on $\R_+^3 \times\cD$ with intensity measure $\Gamma$ given as in \eqref{Gamma}, and that we use the notation \eqref{gamma}.
We are going to show that for all $v\in[0,1]$ there exists a unique c\`adl\`ag process $(Z_t(v), t\geq 0)$ solving 
\begin{equation}\label{ZZ}
\begin{cases}
Z_t(v)= \sum_{s_i=0} w_{t}^i \, \un{(a_i \le v)}  + \sum_{s_i>0} w_{t}^i \, \un{(a_i \le v)} \un{(u_i \le g(Z_{s_i-}(1)))},\;t>0,\\
Z_0=F(v).
\end{cases}
\end{equation}
Then we are going to construct a PPP $\cN$ with intensity $dr\otimes c_\alpha\, \un{(x>0)} \, x^{-1-\alpha}\, dx\otimes dy$ such that, for all $v\in[0,1]$, $Z$ is solution of \eqref{(2.7)} .
Observe that 
$$
Z^0_t(v) := \sum_{s_i=0} w_{t}^i \, \un{(a_i \le v)}  
$$
is well defined when $t>0$ and tends to $F(v)$ as $t\downarrow 0$. Therefore, for each $v$ the process $(Z^0_t(v), t\ge 0)$ is a CSBP started from $F(v)$. If $F(v)=v$ for all $v\in[0,1]$, then $(Z^0_t(v), t\ge 0, v\in[0,1])$ is the whole flow of CSBP.

\subsection{The Pitman-Yor Type Representation with Predictable Random Immigration}

We start by replacing the immigration rate $(g(Z_{s-}(1)))_{s>0}$ in the right-hand side of \eqref{ZZ} with a generic $(\cF_t)$-predictable process $(V_s)_{s\geq 0}$, that we assume to satisfy 
\begin{equation}\label{E:hyp sur V}
V_t\geq 0 \text{ and } \int_0^t \E(V_s) \, ds<+\infty \ \ \forall t\geq 0;
\end{equation} 
this will be useful when we perform a Picard iteration in the proof of existence of solutions to \eqref{(2.7)} and \eqref{ZZ}. Then we consider
\begin{equation}
\label{eq:ex5_1}
\begin{cases}
\ZV_t(v) := \sum_{s_i=0} w_{t}^i \, \un{(a_i \le v)} +
\sum_{s_i>0} w_{t}^i \, \un{(a_i \le v)}\un{(u_i \le V_{s_i})},\qquad t>0, \ v\in[0,1],\\
\ZV_0(v):=F(v),\qquad \ v\in[0,1].
\end{cases}
\end{equation}

Then we want to show that there is a noise $\mathcal N$ on $(\Omega, \mathcal G, \mathcal G_t, \P)$ such that $\ZV$ is a solution of an equation of the type \eqref{(2.7)}.

\subsubsection{Definition of $\cN$}

\begin{figure}
\begin{center}
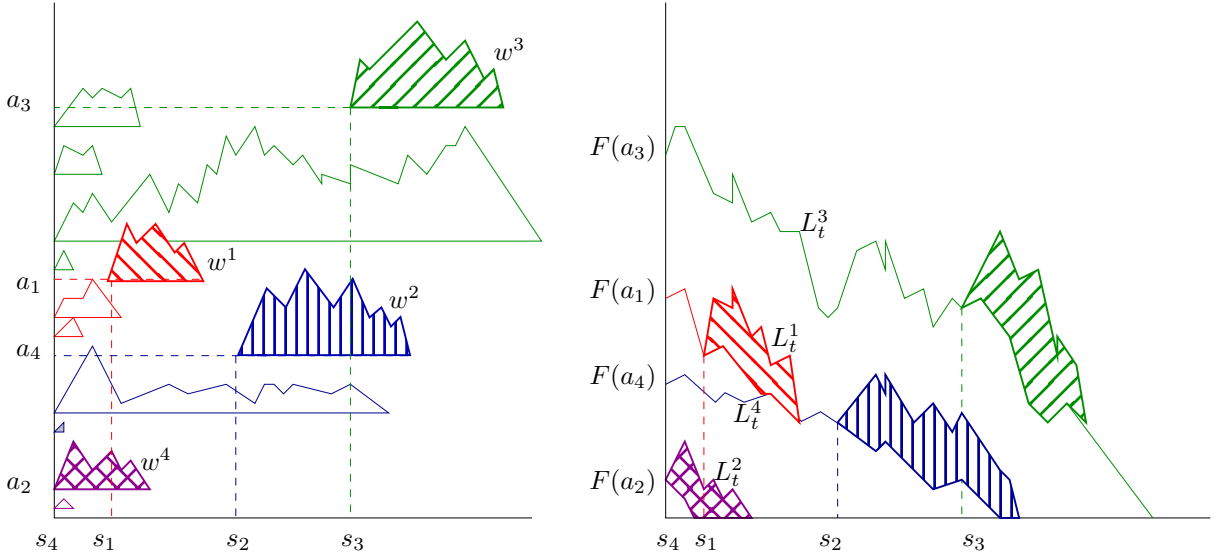
\caption{ Definition of $\mathcal N$. On the left-hand side we represent the  point process $(s_i,w^i, a_i)$. Observe that $s_4=0$ while $s_1,s_2,s_3>0$. On the right-hand side we show how the $w^i$ are combined to construct the noise $\cN$. The line $L^3_t$ for instance is the sum of all the excursions $w^i$ with $a_i \le a_3$. The excursion $w^3$ is then added on top of this line.
}
\label{F:RKconstr}
\end{center}
\end{figure}
Let us consider a family of independent random variables $(U_r)_{r\geq 0}$ such that $U_r$ is uniform
on $[0,1]$ for all $r\geq 0$. We also assume that  $(U_r)_{r\geq 0}$ are 
independent of the PPP $(s_i, u_i,a_i,w^i)$.
Then, for all atoms $(s_i, u_i,a_i,w^i)$ in the above PPP, we define the following point process $\cN^i:=(r_j^i,x_j^i,y_j^i)_{j\in J^i}$: 
\begin{enumerate}
\item $(r_j^i)_{j\in J^i}$ is the family of jump times of $r\mapsto w^i_r$; 
\item for each $r_j^i$ we set
\begin{equation}\label{xy}
x_j^i := w^i_{r_j^i}-w^i_{r_j^i-}, \qquad
y_j^i := w^i_{r_j^i-} \, \cdot U_{r_j^i}.
\end{equation}
\end{enumerate}
Almost surely the sequences $\{ (r_j^i)_j \}_{j\in J^i}$ are disjoint, i.e. $\{r_j^{i_1}\}_{j\in J^{i_1}}\cap \{r_j^{i_2}\}_{j\in J^{i_2}}=\emptyset$ if $i_1\ne i_2$,
 therefore each $\cN^i$ uses a separate family of $(U_r)_{r\in(r_j^i)_{j\in J^i}}$.
We note that $ \cN^i$ is not expected to be a Poisson point process. 
Almost surely we have $a_i\ne a_j$ for all $i\ne j$.  For each $k \in \N$  we set
\begin{equation}\label{L}
\begin{array}{l}  L^k_0 := F(a_k) \text{ and }
L^k_t:=\sum_{a_i< a_k , u_i\le V_{s_i}} w^i_{t-}, \qquad t> 0, \\
 L^\infty_t := \sup_k L^k_t , \qquad t\ge 0. \end{array}
\end{equation}
We consider a PPP $\cN^\circ=(r^\circ_j,x^\circ_j,y^\circ_j)_j$ with intensity measure $\nu$ given by \eqref{nu} and independent of $( (s_i, u_i, a_i, w^i)_i, (U_r)_{r\geq 0}, (V_t)_{t\geq 0}$.
We set for any non-negative measurable $f=f(r,x,y)$ 
\begin{equation}\label{cN}
\begin{split}
 \int f \, d\cN :=  & \sum_{k}  \un{(u_k \le V_{s_k})} \int f(r,x,y+L^k_{r}) \, 
\cN^k(dr, dx,dy) \\ & + \int f(r,x,y+L^\infty_{r}) \, \cN^\circ(dr, dx,dy).
\end{split}
\end{equation}
The filtration we are going to work with is
\[
\cF_t:=\sigma\big((s_i, u_i, a_i, w^i_r, U_r, V_r), (r^\circ_i,x^\circ_i,y^\circ_i)_i : r\leq t, s_i\leq t, r^\circ_i\leq t), \qquad t\geq 0.
\] 
We are going to prove the following
\begin{prop}\label{N}
$\cN$ is a PPP with intensity $\nu(dr , dx , dy) = dr\otimes c_\alpha\, x^{-1-\alpha}\, dx\otimes dy$.
\end{prop}

\begin{proof}
For $f=f(r,x,y)\geq 0$ we now set
\[
I(t):= \sum_k \un{(u_k \le V_{s_k})} \int_{]0,t]\times\R_+\times\R_+} f(r,x,y+L^k_{r}) \, \cN^k(dr, dx,dy).
\]
Since $w_t^i=0$ if $s_i\geq t$, $V$ is predictable and we can write 
\[
L^k_t:=\sum_{s_i=0} \un{(a_i<a_k)} \, w^i_{t-} + \sum_{s_i>0 , u_i\le V_{s_i}} \un{(a_i<a_k )}\, w^i_{t-},
\]
then we obtain that $(L^k_{\cdot})_k$ is predictable. Hence, $I(t)$ is $\cF_t$-measurable and for $0\leq t<T$
\[
\begin{split}
& \E\left( \left. I(T)-I(t) \,\right| \, \cF_t \right) = \E\left( \left. \sum_k \un{(u_k \le V_{s_k})} \int_{]t,T]\times\R_+\times\R_+} f(r,x,y+L^k_{r}) \, \cN^k(dr, dx,dy) \,\right| \, \cF_t \right)
\\  &=  \E\left( \left. \sum_k \un{(u_k \le V_{s_k})} \un{(s_k \le t)} \int_{]t,T]\times\R_+\times\R_+} f(r,x,y+L^k_{r}) \, \cN^k(dr, dx,dy) \,\right| \, \cF_t \right) \\ & \qquad+
 \E\left( \left. \sum_k \un{(u_k \le V_{s_k})}\un{(s_k > t)} \int_{]t,T]\times\R_+\times\R_+} f(r,x,y+L^k_{r}) \, \cN^k(dr, dx,dy) \,\right| \, \cF_t \right)
\end{split}
\]

We will need the following two facts: 
\begin{enumerate}
\item Conditionally on $w^k_t$ and $s_k \le t$  the process $w^k_{\cdot +t}$ has law $P_{w^k_t}$ (this follows for instance from \eqref{bbQ}). 
\item Let $(w_t, t\ge0)$ be a CSBP started from $w_0$ with law $\P_{w_0}$. Let $\cM =(r_i,x_i, y_i)$ be a point process which is defined from $w$ and a sequence of i.i.d. uniform variables on $[0,1]$ as $\cN^k$ is constructed from $w^k$ and the $U_{r^i_j}.$  
Then
for any positive function $f$ 
$$
\E\left[\int_{[0,T] \times \R_+\times \R_+} f(r,x,y) \cM(dr,dx,dy) \right] = \E_{w_0} \left[\int_{[0,T] \times \R_+\times \R_+} f(r,x,y) \mathbf{1}_{y\le w_{r-}} \nu(dr,dx,dy) \right].
$$ 
\end{enumerate}

Let us start with the case $s_k \le t$. Using the above facts we see that
\begin{align*}
&\E\left( \left. \sum_k \un{(u_k \le V_{s_k})}\un{(s_k \le t)} \int_{]t,T]\times\R_+\times\R_+} f(r,x,y+L^k_{r}) \, \cN^k(dr, dx,dy) \,\right| \, \cF_t \right)
\\ &= \E\left( \left. \sum_k \un{(u_k \le V_{s_k})}\un{(s_k \le t)}  \E \left[ \left. \int_{]t,T]\times\R_+\times\R_+} f(r,x,y+L^k_{r}) \, \cN^k(dr, dx,dy)\,\right| w^k_t, L^k_{.} \right]   \, \right| \, \cF_t \right)
\\ &= \E\left( \left. \sum_k \un{(u_k \le V_{s_k})}\un{(s_k \le t)}  \int_{]t,T]\times\R_+\times\R_+} \un{[L^k_{r},L^k_{r}+w^k_{r-} [}(y) \, f(r,x,y) \, dr \, \frac{c_\alpha}{x^{1+\alpha}}\,dx\, dy\,\right| \, \cF_t \right)  
\end{align*}

Let us now consider the case $s_k>t.$ 
\begin{align*}
& \E\left( \left. \sum_k \un{(u_k \le V_{s_k})}\un{(s_k > t)} \int_{]t,T]\times\R_+\times\R_+} f(r,x,y+L^k_{r}) \, \cN^k(dr, dx,dy) \,\right| \, \cF_t \right)
\\ &= \lim_{\epsilon \to 0} \E\left( \left. \sum_k \un{(u_k \le V_{s_k})}\un{(s_k > t)} \int_{]t,T]\times\R_+\times\R_+} f(r,x,y+L^k_{r})  \un{(s_k+\epsilon<r_-)} \, \cN^k(dr, dx,dy) \,\right| \, \cF_t \right)
\\ &= \lim_{\epsilon \to 0} \E\left( \left. \sum_k \un{(u_k \le V_{s_k})}\un{(s_k > t)}  \E\left[ \left. \int_{]t,T]\times\R_+\times\R_+} f(r,x,y+L^k_{r})  \un{(s_k+\epsilon<r_-)} \, \cN^k(dr, dx,dy) \right| w^k_{s_k+\epsilon} ,L^k_{.} \right] \,\right| \, \cF_t \right)
\\ &= \lim_{\epsilon \to 0} \E\left( \left. \sum_k \un{(u_k \le V_{s_k})}\un{(s_k > t)} \int_{]t,T]\times\R_+\times\R_+} \un{[L^k_{r},L^k_{r}+w^k_{r-} [}(y) \, f(r,x,y) \un{(s_k+\epsilon<r_-)}\, dr \, \frac{c_\alpha}{x^{1+\alpha}}\,dx\, dy\,\right| \, \cF_t \right)  
\\ &= \E\left( \left. \sum_k \un{(u_k \le V_{s_k})}\un{(s_k > t)} \int_{]t,T]\times\R_+\times\R_+} \un{[L^k_{r},L^k_{r}+w^k_{r-} [}(y) \, f(r,x,y) \, dr \, \frac{c_\alpha}{x^{1+\alpha}}\,dx\, dy\,\right| \, \cF_t \right) 
\end{align*}
where we need to introduce the indicator that $r_- >s_k+\epsilon$ to get a sum of CSBP started from a positive initial mass and thus be in a position to apply the above fact.

We conclude that 
\begin{align*}
& \E\left( \left. I(T)-I(t) \,\right| \, \cF_t \right)  = \E\left( \left. \int_{]t,T]\times\R_+\times\R_+} \un{]0,\sup_k L^k_{r} [}(y) \, f(r,x,y) \, dr \, \frac{c_\alpha}{x^{1+\alpha}}\,dx\, dy\,\right| \, \cF_t \right),
\end{align*}

Therefore by the definition \eqref{cN} of $\cN$
\[
\begin{split}
& \E\left( \left. \int_{]t,T]\times\R_+\times\R_+} f(r,x,y) \, 
\cN(dr, dx,dy) \,\right| \, \cF_t \right) 
\\ & = \E\left( \left. \int_{]t,T]\times\R_+\times\R_+} \left(\un{]0,L^\infty_{r} [}(y)+\un{]L^\infty_{r},\infty [}(y) \right) \, f(r,x,y) \, dr \, \frac{c_\alpha}{x^{1+\alpha}}\,dx\, dy\,\right| \, \cF_t \right).
\\ & = \int_{]t,T]\times\R_+\times\R_+} f(r,x,y) \, dr \, \frac{c_\alpha}{x^{1+\alpha}}\,dx\, dy.
\end{split}
\]
By \cite[Theorem II.6.2]{ikeda}, a point process with deterministic compensator is necessarily a Poisson point process, and therefore the proof is complete.
\end{proof}

Proposition \ref{N} tells us how to construct a Poisson noise $\cN$ from the $(s_i,u_i,a_i,w^i)$. Let us now show that $Z$ solves \eqref{(2.7)} with this particular noise.
\begin{prop}\label{eqv} Let  $\ZV$ satisfy \eqref{eq:ex5_1}. Then for all $v\geq 0$, $(Z(v),\cN)$ satisfies $\P$-a.s.
\[
\ZV_t(v) = F(v) + \int_{]0,t]\times\R_+\times\R_+} \un{(y< \ZV_{r-}(v))}
\, x \, \tilde \cN(dr,dx,dy) +I(v)\int_0^{t} V_{s} \, ds,\qquad t\geq 0.
\]
\end{prop}
\begin{proof} Using an idea introduced by Dawson and Li \cite{DawsonLi1}, 
we set for $n\in\N^*$
\begin{equation}\label{Y^n}
\ZV_t^{n}(v) := \sum_{a_i \le v, u_i \le V_{s_i}} w_{t}^i \, \un{(s_i+\frac 1n\leq t)}.
\end{equation}
Note that $\bbQ(\{w_{1/n}>0\})<+\infty$ for all $n\geq 1$, so that 
$\ZV^n_t$ is $\P$-a.s. given by a finite sum of terms. Moreover, by the properties of PPPs,
$((s_i,u_i,a_i,w^i): w_{1/n}^i>0)$ is a PPP with intensity $(\delta_0(ds) \otimes \delta_0(du) \otimes F(da) + ds \otimes du\otimes I(da) )\otimes
\un{(w_{1/n}>0)} \, \bbQ(dw)$.
Moreover $Z^n_t\uparrow Z_t$ as $n\uparrow+\infty$ for all $t\geq 0$. 
Now we can write
\[
\ZV_t^n(v) = M^n_t(v)+J^n_t(v),
\]
with
\begin{equation}\label{J^n}
M^n_t(v):= \sum_{a_i \le v , u_i \le V_{s_i}} (w_{t}^i-w_{s_i+\frac1n}^i) \,\un{(s_i+\frac1n\leq t)},
\qquad J^n_t(v) := \sum_{a_i \le v  , u_i \le V_{s_i}} w_{s_i+\frac1n}^i\, \un{(s_i+\frac1n\leq t)}.
\end{equation}
Let us concentrate on $M^n$ first. We can write, for $s_i+\frac1n\leq t$,
\[
w_{t}^i-w_{s_i+\frac1n}^i = \int_{[s_i+\frac1n ,t]\times \R^+\times \R^+} \un{(y< w^i_{r-})} \, x\, \tilde \cN^i(dr,dx,dy)
\]
where $\cN^i$ is defined in \eqref{xy} and $\tilde \cN^i(dr,dx,dy)$ is the compensated version of $\cN^i$: 
\[
\tilde \cN^i(dr,dx,dy) := \cN^i(dr,dx,dy) - \un{(y< w^i_{r-})} \, \nu(dr,dx,dy),
\]
with $\nu$ defined in~\eqref{nu}.
We set
\[
A_{i,n} := \left\{ (y,r) : \ L^i_{r}\leq y < L^i_{r}+w^i_{r-}\, \un{(s_i+\frac1n\leq r )} \right\}, 
\qquad B^{v}_n:=
\bigcup_{a_i \le v\ , \ u_i \le V_{s_i}} A_{i,n}.
\]
Since $\bbQ(\{w_{1/n}>0\})<+\infty$, only finitely many $\{A_{i,n}\}_i$ such that $u_i\le V_{s_i}$ are non-empty $\P$-a.s and, moreover, the $\{A_{i,n}\}_i$ are disjoint. Then by \eqref{cN}
\[
\begin{split}
 &\int_{]0,t]\times \R^+\times \R^+} \un{A_{i,n}}(y,r) \, x\, \tilde \cN(dr,dx,dy) 
\\ & =   \un{(s_i+\frac1n\leq t)}\int_{[s_i+\frac1n,t]\times \R^+\times \R^+} \un{(y< w^i_{r-})} \, x\, \tilde \cN^i(dr,dx,dy) \\ & =
\left( w^i_t-w^i_{s_i+\frac1n} \right)\, \un{(s_i+\frac1n\leq t )}
\end{split}
\]
so that
\[
\int_{]0,t]\times \R^+\times \R^+} \un{B^v_n}(y,r) \, x\, \tilde \cN(dr,dx,dy)  = \sum_{a_i \le v\ , \ u_i \le V_{s_i}} \left( w^i_t-w^i_{s_i+\frac1n} \right)
\, \un{(s_i+\frac1n\leq r)} = M^n_t(v).
\]
We need first the two  following technical lemmas.

\begin{lemma}\label{martingale} 
For a $(\cF_t)$-predictable bounded process $f_t:\R_+\mapsto\R$ we set
\[
M_t := \int_{]0,t]\times\R_+\times\R_+} f_r(y)\, x \, \tilde \cN(dr,dx,dy), \qquad t\geq 0.
\]
Then we have
\[
\begin{split}
\E\left(\sup_{t\in[0,T]} |M_t|\right) \leq & \ C\left(
\sqrt{\int_{[0,T]\times\R_+} \E(f^2_r(y)) \,dr\, dy}+ \int_{[0,T]\times\R_+} \E(|f_r(y)|) \,dr\, dy\right).
\end{split}
\]
\end{lemma}
\begin{proof}
Recall that $\nu_\alpha(dx)=c_\alpha x^{-1-\alpha}\,dx$. We set 
\[
J_{1,t} := \int_{]0,t]\times\R_+\times\R_+} f_r(y)\, \un{(x\leq 1)} \, x \, \tilde \cN(dr,dx,dy), \qquad t\geq 0,
\]
\[
J_{2,t} := \int_{]0,t]\times\R_+\times\R_+} f_r(y)\, \un{(x> 1)} \, x \, \tilde \cN(dr,dx,dy), \qquad t\geq 0.
\]
Then, by Doob's inequality,
\[
\begin{split}
\left(\E\left(\sup_{t\in[0,T]} |J_{1,t}|\right)\right)^2 \leq & \ \E\left(\sup_{t\in[0,T]} |J_{1,t}|^2\right) 
\leq 4\int_{]0,1]}c_\alpha \, x^{1-\alpha}\,dx\int_{[0,T]\times\R_+} \E(f^2_r(y)) \,dr\, dy
\end{split}
\]
while
\[
\begin{split}
\E\left(\sup_{t\in[0,T]} |J_{2,t}|\right) \leq & \ 2 \int_{]1,\infty[} c_\alpha \, x^{-\alpha}\,dx\int_{[0,T]\times\R_+} \E(|f_r(y)|) \,dr\, dy.
\end{split}
\]
\end{proof}

\begin{lemma}\label{LLL} \ 
		\begin{enumerate}
		\item	$\lim_{n\to\infty}\int_{\cE} (z_{\frac1n})^2 \, \un{(z_{\frac1n}\leq 1)} \, \bbQ(dz)=0.$
		\item	$\lim_{n\to\infty}\int_{\cE} z_{\frac1n}\, \un{(z_{\frac1n}\geq 1)} \, \bbQ(dz)=0.$
		\end{enumerate}
	\end{lemma}
	\begin{proof}
		First recall from \eqref{Qmom} that $\int_{\cE} z_{\frac1n}\, \bbQ(dz)=1$ for all $n$.
	The proof of (1) is based on the estimate $\frac{1}{e} x \leq 1-e^{-x}$ for $x\in [0,1]$ which follows from differentiating both sides. Of course, the inequality also implies that
		\begin{align*}
			x^2 \mathbf 1_{(x\leq 1)} \leq e x(1-e^{-x}),\quad x\geq 0.
		\end{align*}
		We apply this estimate to the excursion measure:
		\begin{align}\label{v}
		\begin{split}
			\int_{\cE} (z_{\frac1n})^2 \, \un{(z_{\frac1n}\leq 1)} \, \bbQ(dz)
			&\leq e \int_{\mathcal E}  z_{\frac 1 n } (1-e^{-z_{\frac 1 n}})\, \bbQ(dz)\\
			&=e\left(\int_{\mathcal E} z_{\frac 1 n }\, \bbQ(dz)-\int_{\mathcal E} z_{\frac 1 n} e^{-z_{\frac 1 n}}\, \bbQ(dz)\right).
		\end{split}
		\end{align}
		Next, by \eqref{laplace},
		\begin{align*}
			\int_{\mathcal E}z_{\frac 1 n} e^{-z_{\frac 1 n}}\, \bbQ(dz)&=\frac{d}{d \lambda} u_{1/n}(\lambda)\Big|_{\lambda=1}=\frac{1}{\left(1+(\alpha-1)  \frac{1}{n}\right)^{\alpha/(\alpha-1)}}\stackrel{n\to\infty}{\rightarrow}1,
		\end{align*}
		so that (\ref{v}) combined with $\int_{\cE} z_{\frac1n}\, \bbQ(dz)=1$ proves (1). For (2) we use that $x \mathbf 1_{(x>1)}\leq \frac{e}{(e-1)} x(1-e^{-x})$ to get
		\begin{align*}
			\int_{\mathcal E} z_{\frac 1 n} \mathbf 1_{(z_{\frac 1 n}> 1)}\, \bbQ(dz) \leq \frac{e}{e-1}\int_{\mathcal E} z_{\frac 1 n}(1-e^{-z_{\frac 1 n}}) \, \bbQ(dz)
		\end{align*}
		which goes to zero as argued above.
	\end{proof}

\begin{lemma}\label{3.4}
For all $v\geq 0$ and $T\geq 0$ we have
\[
\lim_{n\to\infty}\E\left( \sup_{t\in[0,T]} |\ZV_t(v)-\ZV_t^n(v)|\right)=0,
\]
$(\ZV_t(v))_{t\geq 0}$ is $\P$-a.s. c\`adl\`ag and $\P$-a.s.
\[
\ZV_t(v) = F(v) + \int_{]0,t]\times\R_+\times\R_+} \un{(y< \ZV_{r-}(v))}
\, x \, \tilde \cN(dr,dx,dy) +I(v)\int_0^{t} V_{s} \, ds.
\]
\end{lemma}

\begin{proof}
We have obtained above the representation
\begin{equation}\label{repY^n}
\ZV_t^n(v) = \int_{]0,t]\times\R_+\times\R_+} \un{B^v_n}(y,r) 
\, x \, \tilde \cN(dr,dx,dy) + J^n_t(v).
\end{equation}
First, let us note that
\[
B^v_n \subset B^v:=\bigcup_{a_i\leq v\ , \ u_i \le V_{s_i}} A_i, \qquad
A_i:= \left\{ (y,r) : \ L^i_{r}\leq y < L^i_{r}+w^i_{r-} \right\},
\]
and moreover
\[
B^v\setminus B^v_n = \bigcup_{a_i\leq v}\left\{ (y,r) : \ L^i_{r}\leq y < L^i_{r}+w^i_{r-} \, \un{(s_i+\frac1n> r)} \right\}
\]
and the latter union is disjoint.
If we set
\[
M_t(v) := \int_{]0,t]\times\R_+\times\R_+} \un{B^v}(y,r) \, x \, \tilde \cN(dr,dx,dy),  
\]
then
\[
\begin{split}
& M_t(v)-M_t^n(v) = \int_{]0,t]\times\R_+\times\R_+} \un{B^v\setminus B^v_n}(y,r) \, x \, \tilde \cN(dr,dx,dy)
\end{split}
\]
and by Lemma \ref{martingale} 
\begin{eqnarray}
\nonumber
\lefteqn{\frac1C\,\E\left(\sup_{t\in[0,T]} |M_t-M^n_t|\right) \leq
 \sqrt{\int_0^T \E(\un{B^v\setminus B^v_n}(y,r)) \, dr\, dy}+ \int_0^T \E(\un{B^v\setminus B^v_n}(y,r)) \, dr\, dy}
\\ \label{excep5_4}
&=& \sqrt{\E\left( \sum_{a_i\leq v ,  u_i \le V_{s_i}} \int_{s_i\wedge T}^{(s_i+\frac1n)\wedge T} w^i_r \, dr \right) }+
\E\left( \sum_{a_i\leq v ,  u_i \le V_{s_i}} \int_{s_i\wedge T}^{(s_i+\frac1n)\wedge T} w^i_r \, dr \right).
\end{eqnarray}
Then we get  
\begin{eqnarray*}
\nonumber
& & \E\left( \sum_{a_i\leq v ,  u_i \le V_{s_i}} \int_{s_i\wedge T}^{(s_i+\frac1n)\wedge T} w^i_r \, dr \right)
\\ &=&  \E\left( \sum_{a_i\leq v} \un{(s_i=0)}\int_{s_i\wedge T}^{(s_i+\frac1n)\wedge T} w^i_r \, dr \right)
  +\E\left( \sum_{a_i\leq v ,  u_i \le V_{s_i}} \un{(s_i>0)}\int_{s_i\wedge T}^{(s_i+\frac1n)\wedge T} w^i_r \, dr \right)\\
\nonumber
&=& \int_0^{T\wedge \frac1n} F(v)  \int_{\cE } w_r \bbQ(dw) dr+\int_0^T \E\left(V_s I(v)\right) \int_s^{(s+\frac1n) \wedge T} \int_{\cE } w_{r-s} \bbQ(dw) dr ds\\
\nonumber
&=& F(v)\left(T\wedge\frac1n\right)
+\int_0^T \E\left(V_sI(v)\right)\left(T\wedge\left(s+\frac1n\right)-s\right) ds,
\end{eqnarray*}
where the last equality follows by~\eqref{Qmom}. By our assumptions on $V$ the right hand side in the above display converges to $0$,
 as  $n\to\infty$. Hence \eqref{excep5_4} also  converges to $0$,
 as  $n\to\infty$. 
Let us now deal with $(J^n_t)_{\geq 0}$. Note that we can write
\[
J^n_{t+\frac1n}(v) = \sum_{a_i \le v , u_i \le V_{s_i}} w^i_{s_i+\frac1n} \, \un{(s_i\leq t)}= A_t^n+\sum_{a_i \le v} w^i_{\frac1n} \, \un{(s_i=0)}+I(v)\int_0^{t} V_{s} \, ds, 
\]
where
\[
A^n_t := \sum_{0<s_i\leq t} w^i_{s_i+\frac1n} \, \un{(a_i \le v, u_i \le V_{s_i})} - I(v)\int_0^{t} V_{s} \, ds, 
\] 
and $(A^n_t)_{t\geq 0}$ is a martingale such that $A^n_0=0$. We have by an analog of Lemma \ref{martingale} and its proof
\[
\E\left(\sup_{t\in[0,T]} |A^n_t|\right) \leq 2\sqrt{K_V \int_{\cE} (z_{\frac1n})^2 \, \un{(z_{\frac1n}\leq 1)} \, \bbQ(dz) }
+ 2K_V\int_{\cE} z_{\frac1n} \, \un{(z_{\frac1n}> 1)} \, \bbQ(dz),
\]
where $K_V:=I(v)\int_0^T \E(V_s)\, ds$. The righthand side tends to zero as $n\to\infty$ by Lemma \ref{LLL}.
Analogously
\[
\begin{split}
& \E\left(\left|\sum_{a_i \le v} w^i_{1/n} \, \un{(s_i=0)}-F(v)\right|\right) 
\\ & \leq 2\sqrt{F(1)\int_{\cE} (z_{\frac1n})^2 \, \un{(z_{\frac1n}\leq 1)} \, \bbQ(dz) }
+ 2F(1)\int_{\cE} z_{\frac1n} \, \un{(z_{\frac1n}> 1)} \, \bbQ(dz),
\end{split}
\]
which again tends to $0$ as $n\to\infty$ by Lemma \ref{LLL}. Therefore 
$$
\bbE \big[ \sup_{t\in[0,T]} |\ZV_t(v)-\ZV^n_t(v)| \big] \to 0.
$$
and, passing to a subsequence, we see that a.s.
\begin{equation}\label{uniform}
\sup_{t\in[0,T]} |\ZV_t(v)-\ZV^{n_k}_t(v)| \to 0
\end{equation}
(observe that in fact we don't need to take a subsequence since $\ZV^n_t$ is monotone non-decreasing in $n$).

In particular, a.s. $(\ZV_t(v), t\geq 0)$ is c\`adl\`ag and we obtain 
\[
\ZV_t(v) = F(v) + \int_{]0,t]\times\R_+\times\R_+} \un{B^v}(y,r)
\, x \, \tilde \cN(dr,dx,dy) +I(v)\int_0^{t} V_{s} \, ds.
\]
It remains to prove that a.s. $B^v = \{(y,r): y< \ZV_{r-}(v)\}$. By definition 
a.s.
\[
\ZV_{r-}(v) = \sum_{a_i\leq v, u_i \le V_{s_i}} w^i_{r-}, \qquad r\geq 0.
\]
If $a_i\leq v$ and $u_i \le V_{s_i}$, then $L^i_{r}+w^i_{r-}\leq \ZV_{r-}(v)$, so that $B^v \subset \{(y,r): y< \ZV_{r-}(v)\}$. On the other hand, if $y< Y_{r-}(v)$, then there is one $j$ such that 
\[
\sum_{a_i<a_j , u_i \le V_{s_i} } w^i_{r-} = L^j_{r} \leq y< L^j_{r}+w^j_{r-}.
\]
Therefore we have obtained the desired results.
\end{proof}
The proof of Proposition \ref{eqv} is complete.
\end{proof}

\subsection{Proof of Theorem \ref{2.4}}
Let us first show uniqueness of solutions to \eqref{ZZ}. Let $v=1$. If $(Z^i_t, t\geq 0)$ for $i=1,2$ is a c\`adl\`ag process satisfying \eqref{ZZ} with $v=1$, then taking the difference we obtain 
\[
\E(|Z^1_t-Z^2_t|) \le  I(1)\int_0^t ds \int_\cE z_{t-s} \, \bbQ(dz) \, \E(|g(Z^1_s)-g(Z^2_s)|) = 
 I(1)\int_0^t ds \, \E(|g(Z^1_s)-g(Z^2_s)|),
\]
where the second equality follows by~\eqref{Qmom}.
By the Lipschitz-continuity of $g$ and the Gronwall Lemma we obtain $Z^1=Z^2$ a.s., i.e. uniqueness of solutions to \eqref{ZZ}.

\medskip 
The next step is to use  an iterative Picard scheme in order to construct a solution of \eqref{ZZ} (and thus of  \eqref{(2.7)}). 
Let $v:=1$, and let us set $Z^0_t:=0$ and for all $n\geq 0$
\[
Z^{n+1}_t:=\sum_{s_i=0} w_{t}^i \, \un{(u_i \le 1)} +
\sum_{s_i>0} w_{t}^i \, \un{(u_i \le g(Z_{s_i-}^n))}, \qquad t\geq 0.
\]
By recurrence
and monotonicity of $g$, $Z^{n+1}_t\geq Z^{n}_t$ and therefore a.s. there exists the limit
$Z_t:=\lim_nZ^{n}_t$. 

To show that $Z$ is actually the solution of  \eqref{ZZ} we show first that it is c\`adl\`ag (by proving that the convergence holds in a norm that makes the space of c\`adl\`ag processes on $[0,T]$ complete) and then by proving that \eqref{ZZ} holds almost surely for each fixed $t\ge 0.$ Let us first show that $Z^n$ is a Cauchy sequence for the norm $\| Z\| = \bbE(\sup_{t\in[0,T]} \left| Z_t \right|)$ for which first we set
\[
Z^{n,k}_t:=Z^{n+k+1}_t-Z^{n+1}_t = \sum_{s_i>0} w_{t}^i \, \un{( g(Z_{s_i-}^{n})<u_i \le g(Z_{s_i-}^{n+k}))}.
\]
By an analog of Proposition \ref{eqv} we can construct a PPP $\cN^{n,k}$ with
the intensity 
measure $\un{(r>0)} \, dr\, \otimes \, c_\alpha \, \un{(x>0)}\, x^{-1-\alpha}\,dx \, \otimes \, \un{(y>0)}\, dy$ such that for all $t\geq 0$
\[
Z^{n,k}_t =  \int_0^t \un{(y< Z^{n,k}_{r-})} \, x\, \tilde \cN^{n,k}(dr,dx,dy)+I(1)\int_0^t\left[g(Z_s^{n+k})-g(Z_{s}^{n})\right]\, ds.
\]
Then by the Lipschitz-continuity of $g$ with the Lipschitz constant $L$,  and by Lemma \ref{martingale} 
\[
\bbE\left(\sup_{t\in[0,T]} \left| Z^{n,k}_t \right| \right) \leq 
\sqrt{\int_0^T \bbE\left(Z^{n,k}_s\right) \, ds}+
\int_0^T \bbE\left(Z^{n,k}_s\right) \, ds + I(1)L 
 \int_0^T \bbE\left(Z^{n-1,k}_s\right) \, ds.
\]
We show now that the right hand side in the latter formula vanishes as $n\to+\infty$ uniformly in $k$. Indeed
\[
\E(Z^{0}_t) = 0, \qquad
0\leq \E(Z^{n+1}_t) = F(1) + \int_0^t \E(g(Z^n_s)) \, ds \leq C+L\int_0^t \E(Z^n_s) \, ds.
\]
Then by recurrence $\E(Z^{n+1}_t) \leq Ce^{tL}$ and by monotone convergence we 
obtain that $\E(Z^{n+1}_t)\uparrow \E(Z_t)\leq Ce^{tL}$. By dominated convergence it follows that 
\[
\int_0^T \E(Z^{n}_s) \, ds \rightarrow \int_0^T \E(Z_s) \, ds,
\]
i.e. the sequence $\int_0^T \E(Z^{n}_s) \, ds$ is Cauchy and we conclude that $Z_n \to Z$ in the sense of the above norm and therefore $Z$ is almost surely c\`adl\`ag. The above argument also show that 
\[
Z_t=\sum_{s_i=0} w_{t}^i \, \un{(u_i \le 1)} +
\sum_{s_i>0} w_{t}^i \, \un{(u_i \le g(Z_{s_i-}))}, 
\]
holds almost surely for each fixed $t$ and therefore for all $t \ge 0$,
i.e. $Z$ is a solution of \eqref{ZZ} for $v=1$. Setting $V_s:=g(Z_{s-}(1))$ and applying Proposition \ref{eqv}, we obtain \eqref{ZZ} and the proof of Theorem \ref{2.4} is complete.

\subsection{Proof of Theorem \ref{2.2}}\label{Section:proofs}
Let us start from existence of a weak solution to \eqref{(2.7)}; by Theorem \ref{2.4} we can build a process $(Z_t(v), t\geq 0, v\in[0,1])$ and a Poisson point process $\mathcal N(dr,dx,dy)$ such that \eqref{ZZ} and \eqref{(2.7)} hold. Now, we set
\[
\textbf X_t := \sum_{s_i=0} w_{t}^i \, \delta_{u_i}  + \sum_{s_i>0} w_{t}^i \, \un{(g(Z_{s_i-}(1))>0)} \, \delta_{u_i /g(Z_{s_i-}(1))},
\]
where $\delta_a$ denote the Dirac mass at $a$; by construction it is clear that $X_t(v):=\textbf X_t([0,v])$, for all $v\in[0,1]$, is a solution to \eqref{(2.7)}. It remains to prove that $(\textbf X_t)_{t\geq 0}$ is c\`adl\`ag in the space of finite measures on the space $[0,1]$. By Lemma \ref{3.4}, for all $v\in[0,1]$, $(X_t(v))_{t\geq 0}$ is c\`adl\`ag; by countable additivity, a.s. $(X_t(v))_{t\geq 0}$ is c\`adl\`ag for all $v\in\bbQ\cap[0,1]$; then, by the compactness of $[0,1]$, it is easy to see that $(\textbf X_t)_{t\geq 0}$ is c\`adl\`ag: for instance, a.s. any limit point of $\textbf X_{t_n}$, for $t_n\geq t$ and $t_n\to t$, is equal on each interval $]a,b]$, $a,b\in\bbQ\cap[0,1]$, to $X_t(b)-X_t(a)=\textbf X_t(]a,b])$. Therefore, we have proved that $(\textbf X_t)_{t\geq 0}$ is a solution to \eqref{(2.7)} in the sense of Definition \ref{def:2.1}.

It remains to prove pathwise uniqueness. Let $(\textbf X^i_t)_{t\geq 0}$, $i=1,2$, be two solutions to \eqref{(2.7)} driven by the same Poisson noise $\cN$ and let us set $X^i_t(v):=\textbf X^i_t([0,v])$, $v\in[0,1]$. Let us first consider the case $v=1$: then $(X^i_t(1), t\geq 0)$, $i=1,2$, solves a particular case of the equation considered by Dawson and Li \cite[(2.1)]{DawsonLi2}; therefore, by \cite[Theorem 2.5]{DawsonLi2},
 $\bbP(X^1_t(1)=X^2_t(1), \, \forall \, t\geq 0)=1$. 

Let us now consider $0\leq v<1$; in this case the equation satisfied by $(X^i_t(v), t\geq 0)$ depends on $(X^i_t(1), t\geq 0)$ and therefore the uniqueness result by Dawson and Li does not apply directly.
Instead, we consider the difference $D_t:=X^1_t(v)-X^2_t(v)$ so that the drift terms cancel  since $X^1(1)=X^2(1)$. Hence, $(D_t, t\geq 0)$ can be treated as 
if $g$ were identically equal to 0. The same proof as in \cite{DawsonLi2} shows that $\bbP(X^1_t(v)=X^2_t(v), \, \forall \, t\geq 0)=1$. Finally, since 
a.s. the two finite measures $\textbf X^1_t$ and $\textbf X^2_t$ are equal on each interval $]a,b]$, $a,b\in\bbQ\cap[0,1]$, they coincide.
Therefore, pathwise uniqueness holds for \eqref{(2.7)}.

Finally, in order to obtain existence of a strong solution, we apply the classical Yamada-Watanabe argument, for instance in the general form proved by Kurtz \cite[Theorem 3.14]{kurtz}.

\section{Proof of Theorem \ref{2.7}}

We
consider the immigration rate function $g(x)=\theta x^{2-\alpha}$, $x\geq 0$. Now $g$ is not Lipschitz-continuous, so that Theorem \ref{2.4} does not apply directly. However, by considering $g^n(x)=\theta (x\vee n^{-1})^{2-\alpha}$, we obtain a monotone non-decreasing and Lipschitz continuous function for which Theorem \ref{2.4} yields existence and uniqueness of a solution $(X^n_t(v), t\geq 0, v\geq 0)$ to \eqref{(2.7)}. We now define $ T^0:=0$,
$T^n := \inf\{t>0: \ X^n_t(1) = n^{-1}\}$ and
\[
X_t(v) := \sum_{n\geq 1} X^n_t(v) \, \un{(T^{n-1}\leq t< T^n)}.
\]
By construction, $T_0:=\sup_n T^n$ is equal to $\inf\{s>0: X_s(1)=0\}$, and moreover $X_t(1)=0$ for all $t\geq T_0$.
By pathwise uniqueness, if $n\geq m$ then $X_t^n(v)=X_t^m(v)$ on $\{t\leq T^m\}$, and therefore $(X_t(v), t\geq 0, v\geq 0)$ is a solution to \eqref{(2.7)} for $g(x)=\theta x^{2-\alpha}$  with the desired properties. Pathwise uniqueness follows from the same localisation argument.\\

	To prove that the right-hand side of \eqref{E: from Y to X via SDE} is well-defined, i.e. the denominator is always strictly positive, we are going to apply Lamperti's representation for self-similar Markov process.
		A positive self-similar Markov process of index $w$ is a strong Markov family $(\P^x)_{x>0}$ with coordinate process 			denoted by $(U_t)_{t\geq 0}$ in the Skorohod space of c\`adl\`ag functions with values in $[0,+\infty[$, satisfying
		\begin{equation}\label{05}
			\text{the law of } (cU_{c^{-1/w} t})_{t\geq 0}\text{ under }\P^x \text{ is given by }\P^{cx}
		\end{equation}
		for all $c>0$. John Lamperti has shown in \cite{L} that this property is equivalent to the existence of a L\'evy process $\xi$ such that, under $\P^x$,  the process $( U_{t\wedge T_0})_{t\geq 0}$ has the same law as
		$\big( x \exp\big(\xi_{A^{-1}({tx^{-1/w})}} \big)\big)_{t\geq 0}$, where
		\begin{align*}
			A^{-1}(t):=\inf\{s\geq 0: A_s> t\}\qquad \text{and}\qquad A(t):=\int_0^t \exp\left(\frac{1}{w}\, \xi_s\right) \, ds.
		\end{align*}
	We now use Lamperti's representation to find a surprisingly simple argument for the well-posedness of \eqref{E: from Y to X via SDE}.
	\begin{lemma}
		The right-hand side of \eqref{E: from Y to X via SDE} is well-defined for all $v\in [0,1]$ and $t\geq 0$.
	\end{lemma}
	\begin{proof}
		In Lemma 1 of \cite{BDMZ} it was  shown that, if $L$ is a spectrally positive $\alpha$-stable L\'evy process as in \eqref{4b}, solutions to the SDE
		\begin{equation}\label{F}
			X_t=X_0 + \int_0^t X_{s-} ^{1/\alpha} \,dL_s+\theta \int_0^t X_s^{2-\alpha}\,ds
		\end{equation}
		trapped at zero induce a positive self-similar Markov process of index $1/(\alpha-1)$. The corresponding L\'evy process $\xi$ has been calculated explicitly in \cite[Lemma 2.2]{BDMZ}, but for the proof here we only need that $\xi$ has infinite lifetime  and additionally a remarkable cancellation effect between the time-changes. Since, by Lemma 1 of Fournier \cite{Fournier}, the unique solution to the SDE \eqref{F} for $X_0=1$ coincides in law with the unique solution to 
		\begin{align*}
			X_t = 1 +\int_{]0,t]\times\R_+\times\R_+} \un{(y<{X_{s-}})} \, x\,  \tilde\cN(ds, dx , dy) + \theta\int_0^t X_{s}^{2-\alpha}\, ds,
		\end{align*}
		we see that the total-mass process $(X_t(1))_{t\geq 0}$ and $\left( \exp\left(\xi_{A^{-1}(t)} \right)\right)_{t\geq 0}$ are equal in law up to first hitting $0$. 
 Applying the Lamperti transformation for $t<T_0$ yields
		\begin{align*}
			\bar S(t)&:=\int_0^t X_s(1)^{1-\alpha}\,ds\\
			&=\int_0^t \exp((1-\alpha)\xi_{A^{-1}(s)})\,ds\\
			&=\int_0^{A^{-1}(t)}\exp((1-\alpha)\xi_s)\exp((\alpha-1)\xi_s)\,ds\\
			&=A^{-1}(t)
		\end{align*}
		so that $\bar S$ and $A$ are reciprocal for $t<T_0$. Plugging this identity into the Lamperti transformation yields
			\begin{equation}\label{abd}
				0= X_{T_0}(1)=\lim_{t\uparrow T_0} X_t(1)=\lim_{t\uparrow T_0}\exp(\xi_{A^{-1}(t)})=\lim_{t\uparrow T_0}\exp(\xi_{\bar S(t)}).
			\end{equation}
			 For the second equality we used left-continuity of $X(1)$ at $T_0$ which is due to Section 3 of \cite{L} because the L\'evy process $\xi$ does not jump to $-\infty$. Using that $\xi_t>-\infty$ for any $t\in [0,\infty)$, from\eqref{abd} we see that $\bar S$ explodes at $T_0$, that is $\bar S(T_0)=\infty$. Since $S$ and $\bar S$ only differ by the factor $\alpha(\alpha-1)\Gamma(\alpha)$, it also holds that $S(T_0)=\infty$ so that $X_{S^{-1}(t)}(1)>0$ for all $t\geq 0$.
\end{proof}	
 
		We can now show how to construct on a pathwise level the Beta-Fleming-Viot processes with mutations the measure-valued branching process.

      \begin{proof}[Proof of Theorem \ref{2.7}]

		Suppose $\mathcal N$ is the PPP with compensator measure $\nu$ that drives the strong solution of \eqref{(2.7)} with atoms
$(r_i,x_i,y_i)_{i\in I}\in (0,\infty)\times (0,\infty)\times (0,\infty)$. Then we define a 
new point process on $(0,\infty)\times (0,\infty)\times (0,\infty)$ by
		\begin{align*}
			\mathcal M (ds,d z,d u):=\sum_{i\in I} \delta_{\big\{S(r_i),\frac{ x_i}{X_{S(r_i)-}(1)+x_i \textbf 1_{(y_i\leq X_{S(r_i)-}(1))}},\frac{y_i}{X_{S(r_i)-}(1)}\big\}}(ds,d z,d u).
		\end{align*}
		If we can show that the restriction $\mathcal M_|$ of $\mathcal M$ to $(0,\infty)\times (0,1)\times (0,1)$ is a PPP with intensity measure 
	$\mathcal M_|'(ds,dz,du)=ds \otimes  C_\alpha z^{-2}z^{1-\alpha}(1-z)^{\alpha-1}dz\otimes du$ and furthermore that
		\begin{align*}
			\bigg(R_t(v):=\frac{X_{S^{-1}(t)}(v)}{X_{S^{-1}(t)}(1)}\bigg)_{t\geq 0, v\in [0,1]}
		\end{align*}
		is a solution to \eqref{FVI} with respect to $\mathcal M_|$, then the claim follows from the pathwise uniqueness of \eqref{FVI}.
		\smallskip
		
		\textbf{Step 1:}	We have
		\begin{align*}
			&\quad R_t(v)
			 =\frac{X_{S^{-1}(t)}(v)}{X_{S^{-1}(t)}(1)} = \\
			&=\frac{X_0(v)}{X_0(1)}+\int_0^{S^{-1}(t)}\int_0^\infty\int_0^\infty\Bigg[\frac{X_{r-}(v)+ x\mathbf 1_{(y\leq X_{r-}(v))}}{X_{r-}(1)+ x\mathbf 1_{(y\leq X_{r-}(1))}}-\frac{X_{r-}(v)}{X_{r-}(1)}\Bigg] (\mathcal{N}-\nu)(dr,dx,dy)	\\
			&\quad +\int_0^{S^{-1}(t)}\int_0^\infty\int_0^\infty\Bigg[\frac{X_{r-}(v)+ x\mathbf 1_{(y\leq X_{r-}(v))}}{X_{r-}(1)+ x\mathbf 1_{(y\leq X_{r-}(1))}}-\frac{X_{r-}(v)}{X_{r-}(1)}-\frac{ x\mathbf 1_{(y\leq X_{r-}(v))} }{X_{r-}(1)}\\
			&\quad\quad+\frac{ x\mathbf 1_{(y\leq X_{r-}(1))} X_{r-}(v)}{X_{r-}(1)^2}\Bigg] \nu(dr,dx,dy)\\
			&\quad + \theta v\int_0^{S^{-1}(t)}\frac{1}{X_{r}(1)}X_r(1)^{2-\alpha} \,dr- \theta \int_0^{S^{-1}(t)}\frac{X_r(v)}{X_{r}(1)^2} X_r(1)^{2-\alpha}\,dr\\
			&=v+\int_0^{S^{-1}(t)}\int_0^\infty\int_0^\infty\Bigg[\frac{X_{r-}(v)+ x\mathbf 1_{(y\leq X_{r-}(v))}}{X_{r-}(1)+ x\mathbf 1_{(y\leq X_{r-}(1))}}-\frac{X_{r-}(v)}{X_{r-}(1)}\Bigg] \mathcal{N}(dr,dx,dy)\\
			&\quad + \theta v\int_0^{S^{-1}(t)}X_r(1)^{1-\alpha} \,dr- \theta \int_0^{S^{-1}(t)}\frac{X_r(v)}{X_{r}(1)} X_r(1)^{1-\alpha}\,dr.
		\end{align*}
		To verify the third equality, first note that due to Lemma II.2.18 of \cite{JS} the compensation can be split from the martingale part and then can be canceled by the compensator integral since integrating-out the $y$-variable 		yields	
		\begin{align*}
			&\quad \int_0^{S^{-1}(t)}\int_0^\infty\int_0^\infty\Bigg[-\frac{ x\mathbf 1_{(y\leq X_{r-}(v))} }{X_{r-}(1)}+\frac{ x\mathbf 1_{(y\leq X_{r-}(1))} X_{r-}(v)}{X_{r-}(1)^2}\Bigg] c_\alpha u^{-1-\alpha} dr\,dx\,dy =0.
		\end{align*}
		To replace the jumps governed by the PPP $\mathcal N$ by jumps governed by $\mathcal M$ note that by the definition of $\mathcal M$ we find, for suitable test-functions $h$, the almost sure transfer identity
		\begin{equation}\label{ff}\begin{split}
				& \int_0^{S^{-1}(t)}\int_0^\infty\int_0^\infty h\left(S(r),\frac{x}{X_{S(r)-}(1)+  x\mathbf 1_{(y\leq X_{S(r)-}(1))}},\frac{y}{X_{S(r)-}(1)}\right)\mathcal N(dr,dx,dy)\\
				&=\int_0^{t}\int_0^1\int_0^\infty h(s, z,  u)\,\mathcal M(ds,d z,d u)\end{split}
		\end{equation}
		or in an equivalently but more suitable form
		\begin{equation}
			\label{ee}
			\begin{split}
				&\quad \int_0^{S^{-1}(t)}\int_0^\infty\int_0^\infty h\left(r,\frac{x}{X_{r-}(1)+  x\mathbf 1_{(y\leq X_{r-}(1))}},\frac{y}{X_{r-}(1)}\right)\mathcal N(dr,dx,dy)\\
				&=\int_0^{t}\int_0^1\int_0^\infty h\big(S^{-1}(s), z,  u\big)\,\mathcal M(d s,d z,d u).\end{split}
		\end{equation}		
		Since the integrals are non-compensated we actually defined $\mathcal M$ in such a way that the integrals produce exactly the same jumps. 
		
		Let us now rewrite the equation found for $R$ in such a way that \eqref{ee} can be applied:
		\begin{align*}
			 R_t(v)
			&=v+\int_0^{S^{-1}(t)}\int_0^\infty\int_0^\infty\Bigg[\frac{x\mathbf 1_{(y\leq X_{r-}(v))}X_{r-}(1)-X_{r-}(v) x\mathbf 1_{(y\leq X_{r-}(1))}}{(X_{r-}(1)+ x\mathbf 1_{(y\leq X_{r-}(1))})X_{r-}(1)}\Bigg] \mathcal{N}(dr,dx,dy)\\
				&\quad + \theta \int_0^{S^{-1}(t)}\Big[vX_r(1)^{1-\alpha} -\frac{X_r(v)}{X_{r}(1)} X_r(1)^{1-\alpha}\Big]\,dr\\
			&=v+\int_0^{S^{-1}(t)}\int_0^\infty\int_0^\infty\frac{x}{X_{r-}(1)+ x\mathbf 1_{(y\leq X_{r-}(1))}} \\
			&\quad\quad \times 
			\left[\mathbf 1_{(y\leq X_{r-}(v))}-\frac{X_{r-}(v)}{X_{r-}(1)}\mathbf 1_{(y\leq X_{r-}(1))}\right] \mathcal{N}(dr,dx,dy)\\
				&\quad + \theta \int_0^{S^{-1}(t)}\Big[vX_r(1)^{1-\alpha} -\frac{X_r(v)}{X_{r}(1)} X_r(1)^{1-\alpha}\Big]\,dr.
		\end{align*}
		The stochastic integral driven by $\mathcal N$ can now be replaced by a stochastic integral driven by $\mathcal M$ via \eqref{ee}:
		\begin{align*}
			&\quad R_t(v)
			= v+\int_0^{t}\int_0^1\int_0^\infty  z\bigg[\mathbf 1_{(  uX_{S^{-1}( s)-}(1)\leq {X_{S^{-1}( s)-}(v)})} - \\
			&\quad\quad- R_{S^{-1}( s)-}( v)\mathbf 1_{( uX_{S^{-1}( s)-}(1)\leq {X_{S^{-1}( s)-}(1)})} \bigg]\mathcal{M}(d s,d z,d u)
			+ \theta \int_0^{t}\big[v-R_s(v)\big]\,ds\\
			&=v+\int_0^{t}\int_0^1\int_0^\infty  z\bigg[\mathbf 1_{(u\leq R_{ s-}(v))}- R_{ s-}( v)\mathbf 1_{( u\leq 1)} \bigg]\mathcal{M}(d s,d z,d u)
			 + \theta \int_0^{t}\big[v-R_s(v)\big]\,ds.
		\end{align*}
		By monotonicity in $v$, $R_t(v)\leq 1$ so that the $du$-integral in fact only runs up to $1$ and the second indicator can be skipped:
		\begin{align*}
			R_t(v)=v+\int_0^{t}\int_0^1\int_0^1  z\bigg[\mathbf 1_{(u\leq R_{ s-}(v))}- R_{ s-}( v) \bigg]\mathcal{M_|}(d s,d z,d u)
			 + \theta \int_0^{t}\big[v-R_s(v)\big]\,ds.
		\end{align*}
		This is precisely the equation we wanted to derive.
		\smallskip
		
		\textbf{Step 2:} The proof is complete if we can show that the restriction $\mathcal M_|$ of $\mathcal M$ to $(0,\infty)\times [0,1]\times [0,1]$ is a PPP with intensity $\mathcal M'(ds,dz,du)=ds \otimes C_\alpha z^{-1-\alpha}(1-z)^{\alpha-1}dz\otimes du$. For this sake, we choose a measurable predictable function $W:\Omega\times (0,\infty)\times (0,1)\times (0,1)\to \R$, plug-in the definition of $\mathcal M_|$ and use the compensator measure $\nu$ of $\mathcal N$ to obtain via \eqref{ff}
		\begin{align*}
			&\E\left(\int_0^t\int_0^1\int_0^1 W( s, z, u) \mathcal M_|(d s,d z,d u)\right)
			=\E\left(\int_0^t\int_0^1\int_0^\infty \mathbf 1_{(u\leq 1)} W( s, z, u) \mathcal M(d s,d z,d u)\right)\\
			&=\E\bigg(\int_0^{S^{-1}(t)}\int_0^\infty\int_0^\infty \mathbf 1_{\big(\frac{y}{X_{S(r)-}(1)}\leq 1\big)}\\
			&\quad \quad\times W\bigg(S(r),\frac{x}{X_{S(r)-}(1)+ x\mathbf 1_{(y\leq X_{S(r)-}(1))}},\frac{y}{X_{S(r)-}(1)}\bigg) \mathcal N(dr,dx,dy)\bigg)
		\end{align*}
		which, by predictable projection and change of variables, equals 
		\begin{align*}
			&\quad\E\bigg(\int_0^{S^{-1}(t)}\int_0^\infty\int_0^\infty \mathbf 1_{\big(\frac{y}{X_{S(r)}(1)}\leq 1\big)} \\
			&\quad\quad \times W\bigg(S(r),\frac{x}{X_{S(r)}(1)+ x\mathbf 1_{(\frac{y}{X_{S(r)}(1)}\leq  1)}},\frac{y}{X_{S(r)}(1)}\bigg)c_\alpha x^{-1-\alpha}\,dr\, dx\, dy\bigg).
			\end{align*}
		Now we substitute the three variables $r, x, y$ (in this order), using $C_\alpha=\frac{1}{\alpha(\alpha-1)\Gamma(\alpha)}c_\alpha$ for the substitution of $r$ and the identity
		\begin{align*}
			\int_0^\infty g\left(\frac{x}{a+x}\right)x^{-1-\alpha}\,dx=a^{-\alpha}\int_0^1 g(z) z^{-1-\alpha}(1-z)^{\alpha-1}\,dz
		\end{align*}		
		for the substitution of $x$ to obtain
		\[
		\begin{split}
&\quad			\E\left(\int_0^t\int_0^1\int_0^1W( s, z, u) \mathcal M_|(d s,d z,d u)\right)
\\ & =\E\left(\int_0^{t}\int_0^1\int_0^1 W(s,z,u) C_\alpha z^{-1-\alpha}(1-z)^{\alpha-1}ds\, dz\, du\right).
		\end{split}
		\]
		Now it follows from Theorems II.4.8 and combined with the definitions of $c_\alpha, C_\alpha$ that $\mathcal M_|$ is a PPP with intensity $ds \otimes  C_\alpha z^{-2}z^{1-\alpha}(1-z)^{\alpha-1}dz\otimes du$.
	\end{proof}


\section{Proof of Theorem \ref{T:main} }\label{111}
  Let us briefly outline the strategy for the proof: In order to show that the measure-valued process $\textbf Y$, $\P$-a.s., does  not posess times $t$ for which $\textbf Y_t$ has
finitely many atoms, by Theorem \ref{2.7} it suffices to show that $\P$-a.s. the same is true for the measure-valued branching process $\textbf X$. In order to achieve this, it suffices 
to deduce the same property for the Pitman-Yor type representation up to extinction, i.e. we need to show that 
		\begin{equation}\label{hjj}
		\bbP\big( \#\{ v\in\, ]0,1]:  Z_t(v)-Z_t(v-)>0\}=\infty, \ \forall \, t\in\,]0,T_0[ \big) =1.
		\end{equation}
Interestingly, this turns out to be easier due to a comparison property that is not available for $\textbf Y$.\\

	We start the proof with a technical result on the covering of a half line by the shadows of a Poisson point process defined on some probability space $(\Omega, \mathcal G, \mathcal G_t, P)$. Suppose $(s_i,h_i)_{i\in I}$ are the points of a Poisson point process $\Pi$ on $(0,\infty)\times (0,\infty)$ 
with intensity $dt\otimes \Pi'(dh)$. For a point $(s_i,h_i)$ we define the shadow on the half line $\R^+$ by $(s_i,s_i+h_i)$ which is precisely the line segment covered by the shadow of the line segment connecting $(s_i,0)$ and $(s_i,h_i)$ with light shining in a 45 degrees angle from the above left-hand side. Shepp proved that the half line $\R^+$ is almost surely fully covered by the shadows induced by the points $(s_i,h_i)_{i\in I}$  if and only if
			\begin{equation}\label{Sh}
			\int_0^1  \exp\left( \int_t^1 (h-t) \,\Pi' (dh)\right)dt=\infty.
		\end{equation}
		The reader is referred to the last remark of \cite{shepp}. For our purposes we need the following variant:
 
\begin{lemma}\label{lemma1}
			Suppose $\Pi$ is a PPP with intensity $dt\,\otimes \,\Pi'(dh)$ and Shepp's condition \eqref{Sh} holds, then
			\begin{align*}
				P\big( \#\big\{  s_i \le t  :  (s_i,h_{i})\in\Pi\text{ and } s_i+h_{i}> t\big\}=\infty, \ \forall\, t>0\big)=1,
			\end{align*}	
			i.e. almost surely every point of $\R^+$ is covered by the shadows of infinitely many line segments.
		\end{lemma}
		
		\begin{proof}
			The proof is an iterated use of Shepp's result for the sequence of restricted Poisson point processes $\Pi_k$ obtained by removing all the atoms $(s_i,h_i)$ with $h_i>\frac{1}{k}$ 
from $\Pi$, i.e. restricting the 
intensity measure to $[0,\frac 1 k]$. Since Shepp's criterion \eqref{Sh} only involves the intensity measure around zero, the shadows of all point processes $\Pi_k$ cover the half line. Consequently, if there is some $t>0$ such that $t$ is only covered by the shadows of finitely 
many points $(s_i,h_i)\in \Pi$, then $t$ is not covered by the shadows generated by $\Pi_{k'}$ for some $k'$ large enough. But this is a contradiction to Shepp's result
 applied to $\Pi_{k'}$.
		\end{proof}
    Now we want to apply Shepp's result to the Pitman-Yor type representation. We want to prove that \eqref{hjj} holds for any $\theta>0$.
		Let us set for all $\epsilon>0$
		\[
		T_\epsilon:= \inf\{t>0: Z_t(1)\leq \epsilon\}.
		\]
		Then it is clearly enough to prove that for all $\epsilon>0$
		\[
		\bbP\big( \#\{ v\in\, ]0,1]:  Z_t(v)-Z_t(v-)>0\}=\infty, \ \forall \, t\in\,]0,T_\epsilon[ \big) =1.
		\]
		In order to connect the covering lemma with the question of exceptional times, we use the comparison property of the Pitman-Yor representation to reduce the problem to the process $Z^\eps$ explicitly defined by
 	\begin{equation}\label{c}
		Z^\eps_t(v) = \sum_{s_i>0} w^i_{t}\, \un{(u_i \le v\,\theta \eps^{2-\alpha})} ,\qquad
		 v \in [0,1],\ t\geq0.
	\end{equation} 
	Setting
	\[
	N_t:=\#\big\{ v\in\, ]0,1]:  Z_t(v)-Z_t(v-)>0\big\}, \qquad N_t^\eps:=\#\big\{ v\in\, ]0,1]:  Z_t^\eps(v)-Z_t^\eps(v-)>0\big\},
	\]
	it is obvious by the definition of $Z$ and $Z^\eps$ that
	\begin{equation}\label{Le}
      	  \bbP( N_t \geq N^\eps_t, \ \forall \, t\in\,]0,T_\epsilon[\} ) =1.
	\end{equation}
	We are now prepared to prove our main result.
	
	\begin{proof}[Proof of Theorem \ref{T:main}]
		Due to \eqref{Le} we only need to show that almost surely $v\mapsto Z^\eps_t(v)$ has infinitely many jumps for all $t>0$ and arbitrary $\eps>0$. To verify the latter, Lemma \ref{lemma1}
will be applied to a PPP defined in the sequel. 
If $\Pi$ denotes the Poisson point process with atoms $(s_i,w^i,u_i)_{i\in I}$ from which 
$Z^\eps_t(v)$ is defined, then we define a new Poisson point process 
$\Pi_l$ via the atoms
		\begin{align*}
			(s_i,h_i,u_i)_{i\in I}:=(s_i,\ell(w^i),u_i)_{i\in I},
		\end{align*}
		where  $\ell(w):=\inf\{t>0: w_t=0\}$ denotes the length of the trajectory $w$. In order to apply Lemma \ref{lemma1} we need the intensity of $\Pi_l$.
 Using the definition of $\Q$ and the Laplace transform duality \eqref{laplace} with the explicit form
\[
		\int\left(1-e^{-\lambda w_t}\right) \bbQ(dw)= \left( \lambda^{1-\alpha} + (\alpha-1)t\right)^{\frac1{1-\alpha}},
\]
		we find the distribution
		\begin{equation}\label{bb}\begin{split}
			\Q(\ell(w)>h)&=  \Q(w_h>0) = \lim_{\lambda\to+\infty}  \Q(1-e^{-\lambda w_h})\\
			&= \lim_{\lambda\to+\infty}  u_h(\lambda) = ((\alpha-1)h)^{\frac{1}{1-\alpha}}.
			\end{split}
		\end{equation}
		Differentiating in $h$ shows that $\Pi_l$ is a Poisson point process on $\R^+\times \R^+\times \R^+$ with intensity measure
		\begin{align*}
			\Pi_l'(dt,dh,du)=dt\otimes ((\alpha-1)h)^{\frac \alpha{(1-\alpha)}}\,dh\otimes du.
		\end{align*}
		Plugging-in the new definitions leads to
		\begin{equation}\label{hj}\begin{split}
			N^\eps_t	
			&= \big(\text{number of non-zero summands of }  Z^\eps_t(1)\big)_{t\geq 0} \\
			&= \big(\#\big\{  s_i \le t  :  (s_i,w_i,u_i)\in\Pi\text{ and } w^i_{t-s_i}\,\un{( u_i \le  \theta \eps^{2-\alpha})}> 0\big\}\big)_{t\geq 0}\\
			&= \big(\#\big\{  s_i \le t  :  (s_i,w_i,u_i)\in\Pi\text{ and }\ell(w_i)>t-s_i, u_i \le \theta \eps^{2-\alpha}\big\}\big)_{t\geq 0}\\
			&= \big(\#\big\{  s_i \le t  :  (s_i,h_i,u_i)\in\Pi_l\text{ and }s_i+h_i>t, u_i \le \theta \eps^{2-\alpha}\big\}\big)_{t\geq 0}.\end{split}
		\end{equation}		
		There is one more simplification that we can do. Let us define $\Pi_{l,\eps}$ as a Poisson point process on $(0,\infty)\times (0,\infty)$ with
 intensity measure
		\begin{equation}\label{S}
			\Pi_{l,\eps}'(dt,dh)=\theta \eps^{2-\alpha}\, dt\otimes   \,(\alpha-1)^{\alpha/(1-\alpha)}h^{\frac \alpha{(1-\alpha)}}\,dh,
		\end{equation}
		then by the properties of Poisson point processes we have the equality in law 
		\[
			\{(s_i,h_i) : (s_i,h_i,u_i)\in\Pi_l\text{ and } u_i \le \theta \eps^{2-\alpha} \}	\stackrel{(d)}{=} \Pi_{l,\eps}.
\]
		Then \eqref{hj} yields 
		\begin{align*}
			\big(N_t^\eps \big)_{t\geq 0} \stackrel{(d)}{=} \big(\#\big\{  s_i \le t  :  (s_i,h_i)\in\Pi_{l,\eps}\text{ and } s_i+h_i > t\big\}\big)_{t\geq 0}.
		\end{align*}
		Now we are precisely in the setting of Shepp's covering results and the theorem follows from Lemma \ref{lemma1} if \eqref{Sh} holds. Shepp's condition can be checked easily for $\Pi_{l,\eps}$ for \eqref{S} independently of $\theta$ and $\eps$.
	\end{proof}

\section{A Proof of Schmuland's Theorem}\label{sec:schmuland}
	In this section we sketch how our lines of arguments can be adopted for the continuous case corresponding to $\alpha=2$. 
The proofs go along the same lines (reduction to a measure-valued branching process and then to an excursion representation for which the covering result can be applied) but are much simpler due to a constant immigration structure. 
The crucial difference, leading to the possibility of exceptional times, occurs in the final step via Shepp's covering results.
	\smallskip
	
	\begin{proof}[Proof of Schmuland's Theorem \ref{schmu}]
		We start with the continuous analogue to Theorem \ref{2.2}.	 Suppose $W$ is a white-noise on $(0,\infty)\times (0,\infty)$, then one can show via the standard Yamada-Watanabe argument that there is a unique strong solution to		
		\begin{equation}\label{eqn:ssde}
		\begin{cases}
			X_t(v)=v+\sqrt 2\int_{]0,t]\times\R_+} \un{(u\leq X_s(v))} \, W(ds,du)+\theta vt,\\
			v\in [0,1],	t\geq 0.
		\end{cases}
		\end{equation}
		 In fact, since the immigration mechanism $g$ is constant, pathwise uniqueness holds. For every $v\in [0,1]$, $(X_t(v))_{t\geq 0}$ satisfies
		 \begin{align*}
		 	X_t(v)=v+\int_0^t \sqrt{2X_s(v)}\, dB_s+v\theta t
		 \end{align*}
		 for a Brownian motion $B$. Recalling \eqref{Feller}, we see that $\eqref{eqn:ssde}$ is a measure-valued process with branching mechanism $\psi(u)=u^2$ and constant-rate 
immigration.
		 
		 \smallskip

		The Pitman-Yor type representation corresponding to Theorem \ref{2.4} looks as follows: in the setting of Section \ref{Sec:RayKnight}, we consider
		a Poisson point process $(s_i, u_i,w^i)_i$ on $\R_+\times\R_+\times\cD$ with intensity measure 
		$(\delta_0(ds)\otimes F(du)+ds\otimes I(du))\otimes \bbQ_s(dw)$, where the excursion measure $\Q$ is defined via the law of the CSBP \eqref{Feller} with branching mechanism
$\psi(\lambda)=\lambda^2$.
Then the analog of Theorem \ref{2.4} is the following:
	 	\begin{equation}\label{RN}
		\begin{cases}
			Z_t(v) = \sum_{s_i=0} w^i_{t}\, \un{( u_i \le v )} + \sum_{0<s_i \le t} w^i_{t-s_i} \un{(u_i \le v \theta )},
			\\
			v\in [0,1],	t\geq 0,
		\end{cases}
		\end{equation}
		can be shown to solve \eqref{eqn:ssde};
		this result, for fixed $v$, goes back to Pitman and Yor \cite{PY}. The calculation \eqref{bb}, now using that	$u_t(\lambda) = \left( \lambda^{-1} +  t\right)^{-1}$ 
		is the unique non-negative solution to
		\begin{align*}
			u'_t(\lambda) = -(u_t(\lambda))^2,	\hspace{.3in} u_0(\lambda) = \lambda,
		\end{align*}
yields $\Q(\ell(w)\in dh)=\frac{1}{h^2}dh$.
		\smallskip
		
		For the analogue for Theorem \ref{2.7} we define now the process
		\begin{align*}
				R_t(v)=\frac{X_{S^{-1}(t)}(v)}{X_{S^{-1}(t)}(1)},
		\end{align*}
		with $S(t) = \int_0^t X_s(1)^{-1}\,ds$. It then follows again from the self-similarity that $R$ is well-defined and from It\=o's formula that $R$ is a standard Fleming-Viot process on $[0,1]$. The arguments here involve a continuous SDE which has been studied in \cite{DawsonLi2}:
		\begin{equation}\label{eqn:fv}
		\begin{cases}
			Y_t(v)=v+\int_0^t \int_0^1 \big[\un{(u\leq Y_s(v))}-Y_s(v)\big] W(ds,du)+\theta \int_0^t \big[v-Y_s(v)\big]\,ds,\\
			v\in [0,1],	t\geq 0,
		\end{cases}
		\end{equation}
		where $W$ is a white-noise on $(0,\infty)\times (0,1)$. It was shown in Theorem 4.9 of \cite{DawsonLi2} that the measure-valued process $\textbf Y$ associated with $(Y_t(v), t\geq 0, v\in[0,1])$ solves the martingale problem for the infinitely many sites model with mutations, i.e.  $\textbf Y$ has generator \eqref{GeneratorFV} with the choice \eqref{E:IAM mutation operator} for $A$.
		\smallskip
		
		Finally, in order to prove Schmuland's Theorem \ref{schmu} on exceptional times it suffices to prove the same result for (\ref{RN}). We proceed again via Shepp's covering arguments as 
we did in Section \ref{111}. The crucial difference is that the immigration is already constant $\theta$ so that \eqref{c}
becomes superfluous. The role of the Poisson point process $\Pi_{l,\eps}$ is played by $\Pi_{\theta,l}$ with 
intensity measure
		\begin{align*}
			\Pi_{\theta,l}'(dt,dh)=dt\otimes \frac{\theta}{h^2}\,dh.
		\end{align*}
		Plugging  into Shepp's criterion \eqref{Sh},
		by  Lemma~\ref{lemma1} and
		\begin{align}\label{inte}
			\int_0^1  \exp\Big(-\theta \log(t)\Big)dt=\int_0^1t^{-\theta}\,dt
		\end{align}
		we find that there are no exceptional times if $\theta\geq 1$.\\
		Conversely, let us assume $\theta<1$. Recalling that for $\theta=0$ the Fleming-Viot process has almost surely finitely many atoms for all $t>0$, we see that the first term in \eqref{RN} almost surely has finitely non-zero summands for all $t>0$. Hence, it suffices to show the existence of exceptional times for which the second term in \eqref{RN} vanishes. Arguing as before, this question is reduced to Shepp's covering result applied to $\Pi_{\theta,l}$:  \eqref{inte} combined with \eqref{Sh} leads to the result.
		
	\end{proof}

\section{Proof of Corollary \ref{cor}}
\label{sec:7}

The fact that  the $(\alpha,\theta)$-Fleming-Viot process  $(\mathbf Y_t,t\ge 0)$ converges in distribution to its unique invariant distribution and that this invariant distribution is not trivial (i.e. it charges measures with at least two atoms) seems to be one of those folklore results for which it is hard to point at a precise reference (however, the existence and unicity of the invariant measure of $(\mathbf Y_t,t\ge 0)$ is proved in \cite{li et al}). Here we sketch an argument that relies on the so-called {\it lookdown} construction of $(\mathbf Y_t,t\ge 0)$.

The lookdown construction was introduced by Donnelly and Kurtz in \cite{dk96} and later expanded in \cite{dk99} by the same authors. The case of Fleming-Viot processes with mutations (in the infinite site model) was treated by Birkner al. \cite{Birkner}. Let us very briefly describe how the lookdown construction works (for more details we refer to \cite{Birkner}).

The idea is to construct a sequence of processes $(\xi_i(t), t\ge 0), i=1,2,\ldots $ which take their values in the type-space $E$ (here $E=[0,1]$). We say that $\xi_i(t)$ is the {\it type} of the {\it level} $i$ at time $t$. The types evolve by two mechanisms : 
\begin{itemize}
\item[-] {\it lookdown events:} with rate $x^{-2}\Lambda(dx)$ a proportion $x$ of lineages are selected by i.i.d. Bernoulli trials. Call $i_1, i_2,\ldots$ the selected levels at a given event at time $t$. Then, $\forall k >1, \xi_{i_k(t)} =\xi_{i_1}(t-)$, that is the levels all adopt the type of the smallest participating level. The type $\xi_{i_k}(t-)$ which was occupying level $i_k$ before the event is pushed up to the next available level. 

\item[-] {\it mutation events:} On each level $i$ there is an independent Poisson point process $(t^{(i)}_j, j\ge 1)$ of rate $\theta$ of mutation events. At a mutation event $t^{(i)}_j$ the type $\xi_i(t^{(i)}_j-)$ is replaced by a new independent variable uniformly distributed on $[0,1]$ and the previous type is pushed up by one level (as well as all the types above him). 
\end{itemize}
The point is then that $$\Xi_t :=\lim_{n\to \infty} \frac1n \sum_{i=1}^n \delta_{\xi_i(t)}$$ exists simultaneously for all $t\ge 0$ almost surely and that $(\Xi_t , t\ge 0) =(\mathbf Y_t, t\ge 0)$ in distribution.

Fix $n \in \N,$  and define a process  $(\pi_t, t\ge 0)$ with values in the partitions of $\{1,2,\ldots\}$ by saying that $i\sim j$ for $\pi^{(n)}_t$ if and only $\xi_i(t)=\xi_j(t).$ It is well known that this is an {\it exchangeable} process. 
Recall from Corollary \ref{C: pure atom} that for each $t\ge 0$ fixed, $\Xi_t$ is almost surely purely atomic. Alternatively this can be seen from the lookdown construction since at a fixed time $t>0$,  the level one has been looked down upon by infinitely many level above since the last mutation event on level one. We can thus write
$$
\Xi_t = \sum a_i(t) \delta_{x_i(t)},
$$
where the $a_i$ are enumerated in decreasing order. It is also known that the sequences $(a_i(t), i\ge 1)$ of atom masses and $(x_i(t), i\ge 1)$ of atom locations are independent.
The $a_i(t)$ are the asymptotic frequencies of the blocks of $\pi(t)$ which are thus in one-to-one correspondence with the atoms of $\Xi_t$. Furthermore the sequence $(x_i(t), i\ge 1)$ converges in distribution to a sequence of i.i.d random variables with common distribution $I$ because all the types that were present initially have been replaced by immigrated types after some time. To see this note that after the first mutation on level 1, the type $\xi_1(0)$ is pushed up to infinity in a finite time which is stochastically dominated by the fixation time of the type at level 1 in a Beta Fleming-Viot without mutation. This also proves the second point of the corollary.

For each $n\ge 1, $ let us consider $\pi^{(n)}(t) = \pi_{|[n]}(t)$ the restriction to $\{1,\ldots,n\}$ of $\pi(t)$. Then, for all $n\ge 1$, the process $(\pi^{(n)}_t, t\ge 0)$ is an irreducible Markov process on a finite state-space and thus converges to its unique invariant distribution. This now implies that $(\pi(t), t\ge 0)$ must also converges to its invariant distribution. By Kingman continuity Theorem (see \cite[Theorem 36]{pitman} or \cite[Theorem1.2]{nath book}) this implies that the ordered sequence of the atom masses $(a_i(t))$ converges in distribution as $t\to \infty$. Because $(x_i(t), i\ge 1)$ also converges in distribution this implies that $\Xi_t$ itself converges in distribution to its invariant measure.

Furthermore it is also clear that the invariant distribution of $(\pi^{(n)}_t, t\ge 0)$ must charge configurations with at least two non-singleton blocks. Since $\pi$ is an exchangeable process, so is its invariant distribution. Exchangeable partitions have only two types of blocks: singletons and blocks with positive asymptotic frequency so this proves that the invariant distribution of $\pi$ charges partition with at least two blocks of positive asymptotic frequency.

\medskip
\noindent \textbf{Aknowledgements:} 
The authors are very grateful for stimulating discussions with Zenghu Li and Marc Yor.
LD was supported by the Fondation Science Math\'ematiques de Paris, LM is partly supported by the Israel Science Foundation. JB, LM and LZ thank the Isaac Newton Institute for Mathematical Sciences, Cambridge, for a very pleasant stay during which part of this work was produced. LM thanks the Laboratoire de Probabilit\'es et Mod\`eles Al\'eatoires for the opportunity to visit it and carry out part of this research there.

\end{document}